\documentclass[11pt,fleqn]{article}
\usepackage{amsfonts}
\usepackage{amssymb}
\usepackage{amsthm}

\usepackage[leqno]{amsmath}
\usepackage{mathtools}

\makeatletter
\newcommand{\leqnomode}{\tagsleft@true}
\newcommand{\reqnomode}{\tagsleft@false}
\makeatother

\usepackage[UKenglish]{babel}
\usepackage{enumerate}
\usepackage{hyperref}
\usepackage{latexsym}
\usepackage{lmodern}
\theoremstyle{plain}
\newtheorem{theorem}{Theorem}[section]
\newtheorem{conjecture}[theorem]{Conjecture}
\newtheorem{lemma}[theorem]{Lemma}
\newtheorem{proposition}[theorem]{Proposition}

\newtheorem{observation}[theorem]{Observation}

\theoremstyle{remark}

\usepackage{tikz}
\usetikzlibrary{arrows,snakes,backgrounds}
\usetikzlibrary{decorations.pathmorphing}
\usepackage{graphicx}
\usetikzlibrary{positioning}
\usepackage{enumitem}

\newtheorem{innercustomgeneric}{\customgenericname}
\providecommand{\customgenericname}{}
\newcommand{\newcustomtheorem}[2]{%
  \newenvironment{#1}[1]
  {%
   \renewcommand\customgenericname{#2}%
   \renewcommand\theinnercustomgeneric{##1}%
   \innercustomgeneric
  }
  {\endinnercustomgeneric}
}

\newcustomtheorem{customtheorem}{Theorem}
\newcustomtheorem{customlemma}{Lemma}

\def\d{\hbox{-}}

\newcommand{\sset}[1]{\left\{#1\right\}}

\def\stackbelow#1#2{\underset{\displaystyle\overset{\displaystyle\shortparallel}{#2}}{#1}}

\usepackage[margin=2.25cm]{geometry}
\usepackage{marvosym}
\renewcommand{\tilde}[1]{\widetilde{#1}}

\def\longbox#1{\parbox{0.85\textwidth}{#1}}

\def\dd{\hbox{-}}

\newcommand*\samethanks[1][\value{footnote}]{\footnotemark[#1]}

\title{New Examples of Minimal Non-Strongly-Perfect Graphs}
\author{Maria Chudnovsky \thanks{Supported by NSF Grant DMS-1763817. This material is based upon work supported by, or in part by, the U.S. Army Research Laboratory and the U. S. Army Research Office under grant number W911NF-16-1-0404.}
\\
Princeton University, Princeton, NJ 08544
\\
\\
Cemil Dibek \samethanks
\\
Princeton University, Princeton, NJ 08544
\\
\\
Paul Seymour\thanks{Supported by AFOSR grant A9550-19-1-0187 and NSF grant DMS-1800053.}
\\
Princeton University, Princeton, NJ 08544}
\date{\today}
\begin{document}\maketitle

\vspace{-0.2cm}

\begin{abstract}
A graph is \emph{strongly perfect} if every induced subgraph $H$ has a stable set that meets every nonempty maximal clique of $H$. The characterization of strongly perfect graphs by a set of forbidden induced subgraphs is not known. Here we provide several new minimal non-strongly-perfect graphs. 
\end{abstract}

\section{Introduction} \label{sec:intro}

All graphs in this paper are finite and simple.  Let $G = (V, E)$ be a graph. For $X \subseteq V(G)$, $G[X]$ denotes the induced subgraph of $G$ with vertex set $X$. We say that $G$ \emph{contains} a graph $H$ if $G$ has an induced subgraph isomorphic to $H$. For a vertex $v \in V(G)$, we let $N_G(v) = N(v)$ denote the set of neighbors of $v$ in $G$. Two disjoint sets $X, Y \subseteq V(G)$ are \emph{complete} to each other if every vertex in $X$ is adjacent to every vertex in $Y$, and \emph{anticomplete} to each other if no vertex in $X$ is adjacent to a vertex in $Y$. We say that $v$ is \emph{complete} (\emph{anticomplete}) to $X \subseteq V(G)$ if $\sset{v}$ is complete (anticomplete) to $X$. 

A \emph{clique} in $G$ is a set of pairwise adjacent vertices, and a \emph{stable set} is a set of pairwise non-adjacent vertices. A \emph{maximal clique} is a clique that is not a subset of a larger clique. A stable set in $G$ is called a \emph{strong stable set} if it meets every nonempty maximal clique of $G$. A \emph{clique cutset} of a graph $G$ is a clique $K$ such that $G \setminus K$ is not connected. A vertex $v \in V(G)$ is a \emph{simplicial vertex} if $N(v)$ is a clique.

A {\em path} in $G$ is an induced subgraph isomorphic to a graph $P$ with $k+1$ vertices $p_0, p_1, \dots, p_k$ and with $E(P) = \{p_ip_{i+1} : i \in \{0,\dots, k-1\}\}$. We write $P = p_0 \dd p_1 \dd \dots \dd p_k$ to denote a path with vertices $p_0, p_1, \dots, p_k$ in order. The \emph{length} of a path is the number of edges in it. A path is {\em odd} if its length is odd, and {\em even} otherwise. For an integer $k \geq 4$, a {\em hole of length $k$} in $G$ is an induced subgraph isomorphic to the $k$-vertex cycle $C_k$, and an {\em antihole of length $k$} is an induced subgraph isomorphic to $\overline{C_k}$. A hole (or antihole) is {\em odd} if its length is odd, and {\em even} if its length is even. A \emph{claw} consists of four vertices, say $a, b, c, d$, with edges $ab$, $ac$, $ad$. A graph is \emph{claw-free} if it contains no induced claw.

A graph is \emph{strongly perfect} if every induced subgraph has a strong stable set. Strongly perfect graphs form a subclass of perfect graphs and have been studied by several authors (\cite{BD, Preissmann, Ravindra, Ravindra3}). A graph $G$ is \emph{minimal non-strongly-perfect} if $G$ is not strongly perfect but every proper induced subgraph of $G$ is. Some results concerning the structure of minimal non-strongly-perfect graphs have been presented in (\cite{AnasOlaru, HMP, Olariu}). In \cite{Ravindra2}, a characterization of claw-free strongly perfect graphs by five infinite families of forbidden induced subgraphs was conjectured, and this was proved by Wang \cite{Wang} in 2006. Recently, a new shorter proof of this characterization was given in \cite{CF-SP}. Nevertheless, the characterization of strongly perfect graphs in general remains open. A conjecture in this direction was presented in 1990.

\begin{conjecture}[\cite{Ravindra2}]
A graph is strongly perfect if and only if it contains no odd holes, no antiholes of length at least six, and none of the Graphs I, II, III, IV, V shown in Figure \ref{fig:minimal_non_strongly_perfect}.
\label{conj:original_conj}
\end{conjecture}

\noindent (Although we refer to them as ``graphs" for convenience, they are actually infinite families of graphs. Also, in all the figures throughout the paper, ``odd" refers to the length of the specified paths, and ``even" refers to the length of the specified holes.)

\vspace{-0.1cm}

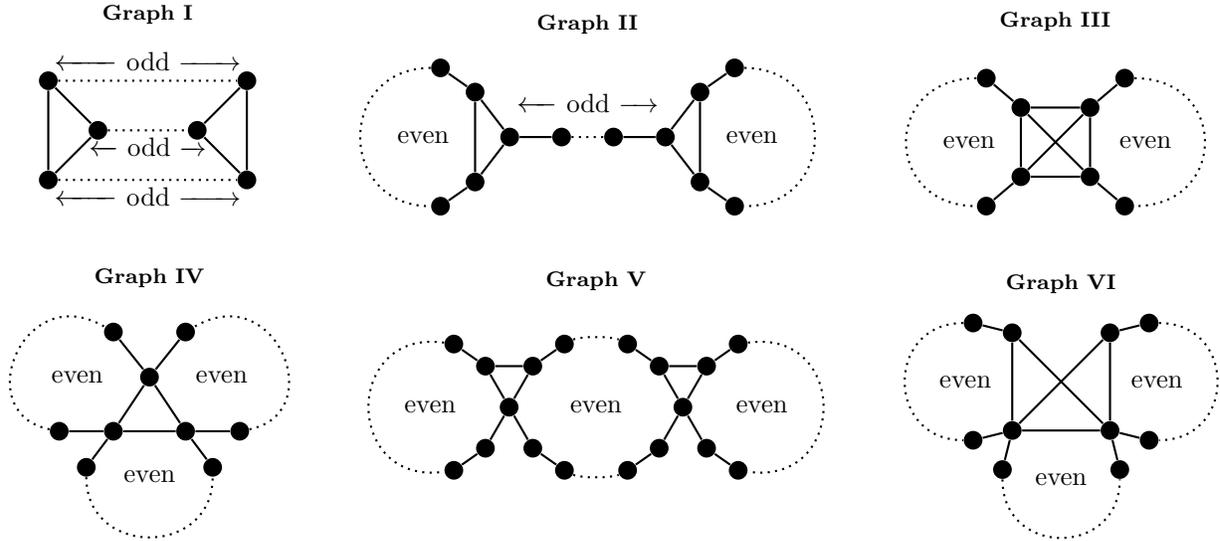
\begin{figure}[h]
\begin{center}
\hspace{0.7cm}
\begin{tikzpicture}[scale=0.22]

\node[inner sep=2.5pt, fill=black, circle] at (0, 0)(v1){}; 
\node[inner sep=2.5pt, fill=black, circle] at (12, 0)(v2){};
\node[inner sep=2.5pt, fill=black, circle] at (3, -3)(v3){}; 
\node[inner sep=2.5pt, fill=black, circle] at (9, -3)(v4){};
\node[inner sep=2.5pt, fill=black, circle] at (0, -6)(v5){}; 
\node[inner sep=2.5pt, fill=black, circle] at (12, -6)(v6){};

\draw[black, dotted, thick] (v1) -- (v2);
\draw[black, dotted, thick] (v3) -- (v4);
\draw[black, dotted, thick] (v5) -- (v6);

\draw[black, thick] (v1) -- (v3);
\draw[black, thick] (v1) -- (v5);
\draw[black, thick] (v3) -- (v5);

\draw[black, thick] (v2) -- (v4);
\draw[black, thick] (v2) -- (v6);
\draw[black, thick] (v4) -- (v6);

\node at (6,1.2) {$\xleftarrow{\makebox[0.6cm]{}}$ \small{odd} $\xrightarrow{\makebox[0.6cm]{}}$};
\node at (6,-4) {$\xleftarrow{\makebox[0.2mm]{}}$ \small{odd} $\xrightarrow{\makebox[0.2mm]{}}$};
\node at (6,-7) {$\xleftarrow{\makebox[0.6cm]{}}$ \small{odd} $\xrightarrow{\makebox[0.6cm]{}}$};

\node at (6,4) {\scriptsize{\textbf{Graph I}}};

\end{tikzpicture}
\hspace{0.8cm}
\begin{tikzpicture}[scale=0.23]

\node[inner sep=2.5pt, fill=black, circle] at (-4, 4)(v1){}; 
\node[inner sep=2.5pt, fill=black, circle] at (-4, -4)(v2){}; 
\node[inner sep=2.5pt, fill=black, circle] at (-2, 2.6)(v3){}; 
\node[inner sep=2.5pt, fill=black, circle] at (-2, -2.6)(v4){};

\node[inner sep=2.5pt, fill=black, circle] at (0, 0)(v5){}; 
\node[inner sep=2.5pt, fill=black, circle] at (3, 0)(v6){}; 
\node[inner sep=2.5pt, fill=black, circle] at (6, 0)(v7){}; 
\node[inner sep=2.5pt, fill=black, circle] at (9, 0)(v8){};

\node[inner sep=2.5pt, fill=black, circle] at (11, 2.6)(v9){}; 
\node[inner sep=2.5pt, fill=black, circle] at (11, -2.6)(v10){}; 
\node[inner sep=2.5pt, fill=black, circle] at (13, 4)(v11){}; 
\node[inner sep=2.5pt, fill=black, circle] at (13, -4)(v12){}; 

\draw[black, dotted, thick] (v1)  .. controls +(-6,0) and +(-6,0) .. (v2);
\draw[black, thick] (v1) -- (v3);
\draw[black, thick] (v2) -- (v4);
\draw[black, thick] (v3) -- (v4);
\draw[black, thick] (v3) -- (v5);
\draw[black, thick] (v4) -- (v5);
\draw[black, thick] (v5) -- (v6);
\draw[black, dotted, thick] (v6) -- (v7);
\draw[black, thick] (v7) -- (v8);
\draw[black, thick] (v8) -- (v9);
\draw[black, thick] (v8) -- (v10);
\draw[black, thick] (v9) -- (v10);
\draw[black, thick] (v9) -- (v11);
\draw[black, thick] (v10) -- (v12);
\draw[black, dotted, thick] (v11)  .. controls +(6,0) and +(6,0) .. (v12);

\node at (-5,0) {\small{even}};
\node at (4.5,2) {$\xleftarrow{\makebox[0.3cm]{}}$ \small{odd} $\xrightarrow{\makebox[0.3cm]{}}$};
\node at (14,0) {\small{even}};

\node at (4.5,6.5) {\scriptsize{\textbf{Graph II}}};

\end{tikzpicture}
\hspace{0.3cm}
\begin{tikzpicture}[scale=0.23]

\node[inner sep=2.5pt, fill=black, circle] at (-4, 3.7)(v1){}; 
\node[inner sep=2.5pt, fill=black, circle] at (-4, -3.7)(v2){}; 
\node[inner sep=2.5pt, fill=black, circle] at (-2, 2)(v3){}; 
\node[inner sep=2.5pt, fill=black, circle] at (-2, -2)(v4){};

\node[inner sep=2.5pt, fill=black, circle] at (2, 2)(v9){}; 
\node[inner sep=2.5pt, fill=black, circle] at (2, -2)(v10){}; 
\node[inner sep=2.5pt, fill=black, circle] at (4, 3.7)(v11){}; 
\node[inner sep=2.5pt, fill=black, circle] at (4, -3.7)(v12){}; 

\draw[black, dotted, thick] (v1)  .. controls +(-6,0) and +(-6,0) .. (v2);
\draw[black, thick] (v1) -- (v3);
\draw[black, thick] (v2) -- (v4);
\draw[black, thick] (v3) -- (v4);
\draw[black, thick] (v3) -- (v9);
\draw[black, thick] (v3) -- (v10);
\draw[black, thick] (v4) -- (v9);
\draw[black, thick] (v4) -- (v10);
\draw[black, thick] (v9) -- (v10);
\draw[black, thick] (v9) -- (v11);
\draw[black, thick] (v10) -- (v12);
\draw[black, dotted, thick] (v11)  .. controls +(6,0) and +(6,0) .. (v12);

\node at (-5,0) {\small{even}};
\node at (5.2,0) {\small{even}};

\node at (0,7) {\scriptsize{\textbf{Graph III}}};

\end{tikzpicture}
\begin{tikzpicture}[scale=0.24]

\node[inner sep=2.5pt, fill=black, circle] at (0.5, 0)(v1){}; 
\node[inner sep=2.5pt, fill=black, circle] at (7.5, 0)(v2){}; 
\node[inner sep=2.5pt, fill=black, circle] at (2, 2)(v3){}; 
\node[inner sep=2.5pt, fill=black, circle] at (6, 2)(v4){};
\node[inner sep=2.5pt, fill=black, circle] at (4, 5)(v5){}; 
\node[inner sep=2.5pt, fill=black, circle] at (-1, 2)(v6){}; 
\node[inner sep=2.5pt, fill=black, circle] at (9, 2)(v7){}; 
\node[inner sep=2.5pt, fill=black, circle] at (2, 7.5)(v8){};
\node[inner sep=2.5pt, fill=black, circle] at (6, 7.5)(v9){}; 

\node[inner sep=2.5pt, fill=white, circle] at (4, 13.3)(v21){}; 
\node[inner sep=2.5pt, fill=white, circle] at (16, 13.3)(v22){}; 

\draw[black, dotted, thick] (v1)  .. controls +(0,-5) and +(0,-5) .. (v2);
\draw[black, dotted, thick] (v6)  .. controls +(-5,1) and +(-5,3) .. (v8);
\draw[black, dotted, thick] (v7)  .. controls +(5,1) and +(5,3) .. (v9);
\draw[black, thick] (v1) -- (v3);
\draw[black, thick] (v2) -- (v4);
\draw[black, thick] (v3) -- (v4);
\draw[black, thick] (v3) -- (v5);
\draw[black, thick] (v4) -- (v5);
\draw[black, thick] (v6) -- (v3);
\draw[black, thick] (v7) -- (v4);
\draw[black, thick] (v8) -- (v5);
\draw[black, thick] (v9) -- (v5);

\node at (4,-0.5) {\small{even}};
\node at (8,5) {\small{even}};
\node at (0,5) {\small{even}};

\node at (4,10.5) {\scriptsize{\textbf{Graph IV}}};

\end{tikzpicture}
\hspace{-0.7cm}
\begin{tikzpicture}[scale=0.21]

\node[inner sep=2.5pt, fill=black, circle] at (-4, 4)(v1){}; 
\node[inner sep=2.5pt, fill=black, circle] at (-4, -4)(v2){}; 
\node[inner sep=2.5pt, fill=black, circle] at (-2, 2.6)(v3){}; 
\node[inner sep=2.5pt, fill=black, circle] at (-2, -2.6)(v4){};

\node[inner sep=2.5pt, fill=black, circle] at (-0.5, 0)(v5){};

\node[inner sep=2.5pt, fill=black, circle] at (3, 4)(v13){}; 
\node[inner sep=2.5pt, fill=black, circle] at (3, -4)(v14){}; 
\node[inner sep=2.5pt, fill=black, circle] at (1, 2.6)(v15){}; 
\node[inner sep=2.5pt, fill=black, circle] at (1, -2.6)(v16){};

\node[inner sep=2.5pt, fill=black, circle] at (12, 2.6)(v9){}; 
\node[inner sep=2.5pt, fill=black, circle] at (12, -2.6)(v10){}; 
\node[inner sep=2.5pt, fill=black, circle] at (14, 4)(v11){}; 
\node[inner sep=2.5pt, fill=black, circle] at (14, -4)(v12){}; 

\node[inner sep=2.5pt, fill=black, circle] at (10.5, 0)(v8){};

\node[inner sep=2.5pt, fill=black, circle] at (9, 2.6)(v17){}; 
\node[inner sep=2.5pt, fill=black, circle] at (9, -2.6)(v18){}; 
\node[inner sep=2.5pt, fill=black, circle] at (7, 4)(v19){}; 
\node[inner sep=2.5pt, fill=black, circle] at (7, -4)(v20){}; 

\node[inner sep=2.5pt, fill=white, circle] at (7, -9)(v21){}; 

\draw[black, dotted, thick] (v1)  .. controls +(-7,1) and +(-7,-1) .. (v2);
\draw[black, thick] (v1) -- (v3);
\draw[black, thick] (v2) -- (v4);
\draw[black, thick] (v3) -- (v5);
\draw[black, thick] (v4) -- (v5);
\draw[black, thick] (v8) -- (v9);
\draw[black, thick] (v8) -- (v10);
\draw[black, thick] (v9) -- (v11);
\draw[black, thick] (v10) -- (v12);
\draw[black, dotted, thick] (v11)  .. controls +(7,1) and +(7,-1) .. (v12);

\draw[black, thick] (v15) -- (v5);
\draw[black, thick] (v16) -- (v5);
\draw[black, thick] (v13) -- (v15);
\draw[black, thick] (v14) -- (v16);

\draw[black, thick] (v17) -- (v8);
\draw[black, thick] (v18) -- (v8);
\draw[black, thick] (v17) -- (v19);
\draw[black, thick] (v18) -- (v20);

\draw[black, dotted, thick] (v13)  .. controls +(1,0.5) and +(-1,0.5) .. (v19);
\draw[black, dotted, thick] (v14)  .. controls +(1,-0.5) and +(-1,-0.5) .. (v20);

\draw[black, thick] (v3) -- (v15);
\draw[black, thick] (v17) -- (v9);

\node at (-5.5,0) {\small{even}};
\node at (5,0) {\small{even}};
\node at (15.5,0) {\small{even}};

\node at (5,8) {\scriptsize{\textbf{Graph V}}};

\end{tikzpicture}
\hspace{0.2cm}
\begin{tikzpicture}[scale=0.26]

\node[inner sep=2.5pt, fill=black, circle] at (-4, 2.5)(v1){}; 
\node[inner sep=2.5pt, fill=black, circle] at (-4, -3.5)(v2){}; 
\node[inner sep=2.5pt, fill=black, circle] at (-2, 2)(v3){}; 
\node[inner sep=2.5pt, fill=black, circle] at (-2, -3)(v4){};

\node[inner sep=2.5pt, fill=black, circle] at (3, 2)(v9){}; 
\node[inner sep=2.5pt, fill=black, circle] at (3, -3)(v10){}; 
\node[inner sep=2.5pt, fill=black, circle] at (5, 2.5)(v11){}; 
\node[inner sep=2.5pt, fill=black, circle] at (5, -3.5)(v12){}; 

\node[inner sep=2.5pt, fill=black, circle] at (-2.5, -5)(v13){}; 
\node[inner sep=2.5pt, fill=black, circle] at (3.5, -5)(v14){}; 

\draw[black, dotted, thick] (v1)  .. controls +(-4.5,0) and +(-4.5,0) .. (v2);
\draw[black, thick] (v1) -- (v3);
\draw[black, thick] (v2) -- (v4);
\draw[black, thick] (v3) -- (v4);
\draw[black, thick] (v3) -- (v10);
\draw[black, thick] (v4) -- (v9);
\draw[black, thick] (v4) -- (v10);
\draw[black, thick] (v9) -- (v10);
\draw[black, thick] (v9) -- (v11);
\draw[black, thick] (v10) -- (v12);
\draw[black, thick] (v4) -- (v13);
\draw[black, thick] (v10) -- (v14);
\draw[black, dotted, thick] (v11)  .. controls +(4.5,0) and +(4.5,0) .. (v12);

\draw[black, dotted, thick] (v13)  .. controls +(0,-4.5) and + (0,-4.5) .. (v14);

\node at (-4.4,-0.5) {\small{even}};
\node at (5.4,-0.5) {\small{even}};
\node at (0.5, -5.5) {\small{even}};

\node at (0.5,4.5) {\scriptsize{\textbf{Graph VI}}};

\end{tikzpicture}
\end{center}
\vspace{-0.7cm}
\caption{Some forbidden induced subgraphs for strongly perfect graphs}
\label{fig:minimal_non_strongly_perfect}
\end{figure}

Later, in 1999, another minimal non-strongly-perfect graph, proposed by Maffray, appeared in a paper of Ravindra \cite{Ravindra3} (see Graph VI in Figure \ref{fig:minimal_non_strongly_perfect}). To the best of our knowledge, this is a complete list of minimal non-strongly-perfect graphs that have appeared in the literature. Here, we extend the list by providing several new infinite families of minimal non-strongly-perfect graphs.


\section{Preliminaries} \label{sec:preliminaries}

We start with a remark about one of the graphs listed in Conjecture \ref{conj:original_conj}. Let $G$ be Graph V in Figure \ref{fig:minimal_non_strongly_perfect}. For $i=1,2$, let $P_i$ be the $u_iv_i$-path in $G$, as shown in Figure \ref{fig:minimal_fifth} (left). It can be checked that $G$ is not minimal non-strongly-perfect unless $P_1$ is odd and $P_2$ is of length one, that is, $u_2$ is adjacent to $v_2$. Therefore, from now on, when we say Graph V, we refer to the graph shown in Figure \ref{fig:minimal_fifth} (right).

\vspace{-0.15cm}

\begin{figure}[h]
\begin{center}
\begin{tikzpicture}[scale=0.22]

\node[inner sep=2.5pt, fill=black, circle] at (-4, 4)(v1){}; 
\node[inner sep=2.5pt, fill=black, circle] at (-4, -4)(v2){}; 
\node[inner sep=2.5pt, fill=black, circle] at (-2, 2.6)(v3){}; 
\node[inner sep=2.5pt, fill=black, circle] at (-2, -2.6)(v4){};

\node[inner sep=2.5pt, fill=black, circle] at (-0.5, 0)(v5){};

\node[inner sep=2.5pt, fill=black, circle] at (3, 4)(v13){}; 
\node[inner sep=2.5pt, fill=black, circle] at (3, -4)(v14){}; 
\node[inner sep=2.5pt, fill=black, circle] at (1, 2.6)(v15){}; 
\node[inner sep=2.5pt, fill=black, circle] at (1, -2.6)(v16){};

\node[inner sep=2.5pt, fill=black, circle] at (12, 2.6)(v9){}; 
\node[inner sep=2.5pt, fill=black, circle] at (12, -2.6)(v10){}; 
\node[inner sep=2.5pt, fill=black, circle] at (14, 4)(v11){}; 
\node[inner sep=2.5pt, fill=black, circle] at (14, -4)(v12){}; 

\node[inner sep=2.5pt, fill=black, circle] at (10.5, 0)(v8){};

\node[inner sep=2.5pt, fill=black, circle] at (9, 2.6)(v17){}; 
\node[inner sep=2.5pt, fill=black, circle] at (9, -2.6)(v18){}; 
\node[inner sep=2.5pt, fill=black, circle] at (7, 4)(v19){}; 
\node[inner sep=2.5pt, fill=black, circle] at (7, -4)(v20){}; 

\node[inner sep=2.5pt, fill=white, circle] at (7, -7.5)(v21){}; 

\draw[black, dotted, thick] (v1)  .. controls +(-7,1) and +(-7,-1) .. (v2);
\draw[black, thick] (v1) -- (v3);
\draw[black, thick] (v2) -- (v4);
\draw[black, thick] (v3) -- (v5);
\draw[black, thick] (v4) -- (v5);
\draw[black, thick] (v8) -- (v9);
\draw[black, thick] (v8) -- (v10);
\draw[black, thick] (v9) -- (v11);
\draw[black, thick] (v10) -- (v12);
\draw[black, dotted, thick] (v11)  .. controls +(7,1) and +(7,-1) .. (v12);

\draw[black, thick] (v15) -- (v5);
\draw[black, thick] (v16) -- (v5);
\draw[black, thick] (v13) -- (v15);
\draw[black, thick] (v14) -- (v16);

\draw[black, thick] (v17) -- (v8);
\draw[black, thick] (v18) -- (v8);
\draw[black, thick] (v17) -- (v19);
\draw[black, thick] (v18) -- (v20);

\draw[black, dotted, thick] (v13)  .. controls +(1,0.5) and +(-1,0.5) .. (v19);
\draw[black, dotted, thick] (v14)  .. controls +(1,-0.5) and +(-1,-0.5) .. (v20);

\draw[black, thick] (v3) -- (v15);
\draw[black, thick] (v17) -- (v9);

\node at (5,6) {$P_1$};
\node at (5,-6) {$P_2$};

\node at (-0.4,-1.5) {$\scriptstyle{u_2}$};
\node at (10.6,-1.5) {$\scriptstyle{v_2}$};

\node at (1, 3.8) {$\scriptstyle{u_1}$};
\node at (9, 3.8) {$\scriptstyle{v_1}$};

\node at (-5.5,0) {\small{even}};
\node at (5,0) {\small{even}};
\node at (15.5,0) {\small{even}};

\end{tikzpicture}
\begin{tikzpicture}[scale=0.23]

\node[inner sep=2.5pt, fill=black, circle] at (-4, 4)(v1){}; 
\node[inner sep=2.5pt, fill=black, circle] at (-4, -4)(v2){}; 
\node[inner sep=2.5pt, fill=black, circle] at (-2, 2.6)(v3){}; 
\node[inner sep=2.5pt, fill=black, circle] at (-2, -2.6)(v4){};

\node[inner sep=2.5pt, fill=black, circle] at (0, 0)(v5){}; 
\node[inner sep=2.5pt, fill=black, circle] at (3, 0)(v6){}; 
\node[inner sep=2.5pt, fill=black, circle] at (6, 0)(v7){}; 
\node[inner sep=2.5pt, fill=black, circle] at (9, 0)(v8){};

\node[inner sep=2.5pt, fill=black, circle] at (11, 2.6)(v9){}; 
\node[inner sep=2.5pt, fill=black, circle] at (11, -2.6)(v10){}; 
\node[inner sep=2.5pt, fill=black, circle] at (13, 4)(v11){}; 
\node[inner sep=2.5pt, fill=black, circle] at (13, -4)(v12){}; 

\node[inner sep=2.5pt, fill=white, circle] at (0, -7.2)(v21){}; 

\draw[black, dotted, thick] (v1)  .. controls +(-6,0) and +(-6,0) .. (v2);
\draw[black, thick] (v1) -- (v3);
\draw[black, thick] (v2) -- (v4);
\draw[black, thick] (v3) -- (v4);
\draw[black, thick] (v3) -- (v5);
\draw[black, thick] (v4) -- (v5);
\draw[black, thick] (v5) -- (v6);
\draw[black, dotted, thick] (v6) -- (v7);
\draw[black, thick] (v7) -- (v8);
\draw[black, thick] (v8) -- (v9);
\draw[black, thick] (v8) -- (v10);
\draw[black, thick] (v9) -- (v10);
\draw[black, thick] (v9) -- (v11);
\draw[black, thick] (v10) -- (v12);
\draw[black, dotted, thick] (v11)  .. controls +(6,0) and +(6,0) .. (v12);

\draw[black, thick] (v4) -- (v10);

\node at (-5,0) {\small{even}};
\node at (4.5,2) {$\xleftarrow{\makebox[0.3cm]{}}$ {\small{odd}} $\xrightarrow{\makebox[0.3cm]{}}$};
\node at (14,0) {\small{even}};

\node at (5,6) {\scriptsize{\textbf{Graph V}}};

\end{tikzpicture}
\end{center}
\vspace{-0.8cm}
\caption{A closer look at Graph V shown in Figure \ref{fig:minimal_non_strongly_perfect}}
\label{fig:minimal_fifth}
\end{figure}
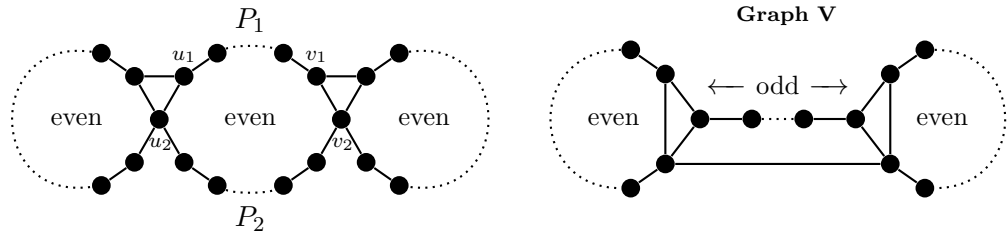

We continue by stating some observations and proving a few lemmas that we will use later.

\begin{lemma}
Let $K$ be a clique cutset in a graph $G$ where $(A, B, K)$ is a partition of $V(G)$ such that $A$ is anticomplete to $B$, and $K$ is a clique. Assume that either $K = \{k\}$ and $k$ is not anticomplete to $A$, or no vertex in $B$ is complete to $K$. Let $H = G[A \cup K]$. If $S$ is a strong stable set in $G$, then $S' = S \cap V(H)$ is a strong stable set in $H$. In particular, if $G$ has a strong stable set $S$ with $v \in K \cap S$ (resp. $v \notin K \cap S$), then $H$ has a strong stable set $S'$ with $v \in K \cap S'$ (resp. $v \notin K \cap S'$).
\label{lem:clique_cutset_sss}
\end{lemma}

\begin{proof}
We prove that if $C$ is a maximal clique in $H$, then $C$ is a maximal clique in $G$. Suppose not. Then there exists a vertex $v \in V(G) \setminus C$ such that $v$ is complete to $C$. Since $C$ is maximal in $H$, it follows that $v \notin V(H)$, and so $v \in B$. Since $B$ is anticomplete to $A$, we deduce that $C \subseteq K$, and so $C = K$ as $K$ is a clique. This is a contradiction in the first case since $C = K = \{k\}$ is not a maximal clique in $H$ as $k$ is not anticomplete to $A$, and a contradiction in the second case since no vertex in $B$ is complete to $K$. Thus, every maximal clique of $H$ is also a maximal clique of $G$.
\end{proof}

\begin{observation}
If $v$ is a simplicial vertex in a graph $G$, and $S$ is a strong stable set of $G \setminus \{v\}$, then either $S$ or $S \cup \{v\}$ is a strong stable set of $G$.
\label{obs:simplicial}
\end{observation}

\begin{lemma}
Let $G$ be a minimal non-strongly-perfect graph. Then, $G$ has no simplicial vertex.
\label{lem:no_simplicial_in_minimal}
\end{lemma}

\begin{proof}
Let $v$ be a simplicial vertex of $G$. By minimality of $G$, the graph $G \setminus v$ is strongly perfect. Now, by Observation \ref{obs:simplicial}, it follows that $G$ has a strong stable set, a contradiction.
\end{proof}

The \emph{basis} of a graph $G$, denoted by $\mathcal{B}(G)$, is the set of all proper induced subgraphs of $G$ with no simplicial vertex. We say that $G$ has a \emph{strong basis} if the graphs in $\mathcal{B}(G)$ are all strongly perfect. Let $G$ be a graph with a strong basis. In view of Lemma \ref{lem:no_simplicial_in_minimal}, if $G$ has a strong stable set, then $G$ is strongly perfect, and if $G$ has no strong stable set, then $G$ is minimal non-strongly-perfect. It is easy to check that the graphs in Figure \ref{fig:graphs_A_1234567} have a strong basis and a strong stable set, and therefore they are strongly perfect.

\begin{observation}
Let $k, m, n \geq 4$ be even. The graphs $A_1, \dots, A_6$ in Figure \ref{fig:graphs_A_1234567} are strongly perfect.
\end{observation}

\vspace{-0.25cm}

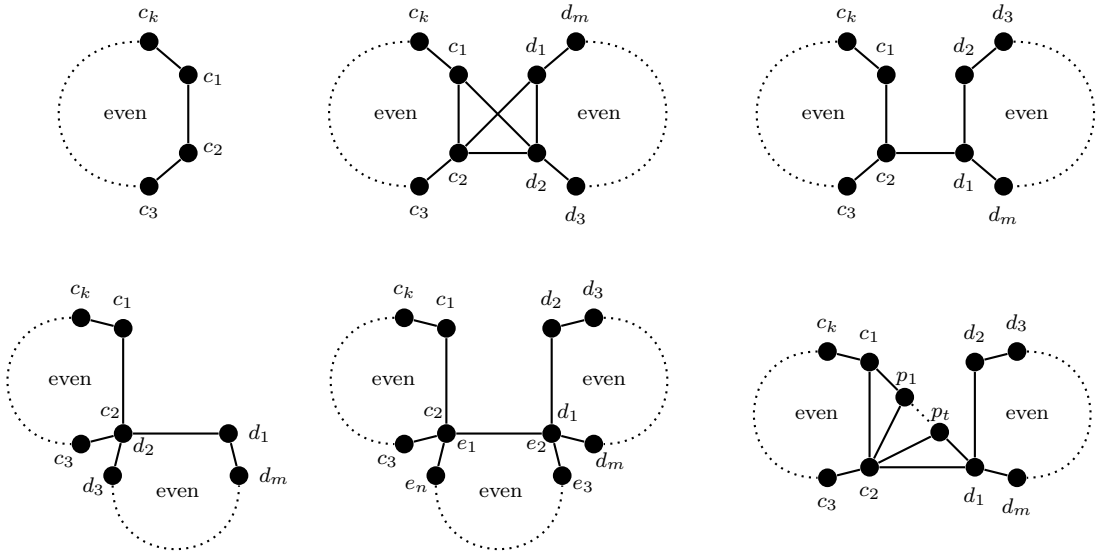
\begin{figure}[h]
\begin{center}
\hspace{0.6cm}
\begin{tikzpicture}[scale=0.26]

\node[label=above:{$\scriptstyle{c_k}$}, inner sep=2.5pt, fill=black, circle] at (-4, 3.7)(v1){}; 
\node[label=below:{$\scriptstyle{c_3}$}, inner sep=2.5pt, fill=black, circle] at (-4, -3.7)(v2){}; 
\node[inner sep=2.5pt, fill=black, circle] at (-2, 2)(v3){}; 
\node[inner sep=2.5pt, fill=black, circle] at (-2, -2)(v4){};

\node at (-0.7,1.8) {$\scriptstyle{c_1}$};
\node at (-0.7,-1.8) {$\scriptstyle{c_2}$};

\node[inner sep=2.5pt, fill=white, circle] at (0, -7.4)(v21){}; 

\draw[black, dotted, thick] (v1)  .. controls +(-6,0) and +(-6,0) .. (v2);
\draw[black, thick] (v1) -- (v3);
\draw[black, thick] (v2) -- (v4);
\draw[black, thick] (v3) -- (v4);

\node at (-5.2, 0) {\scriptsize{even}};

\end{tikzpicture}
\hspace{0.6cm}
\begin{tikzpicture}[scale=0.26]

\node[label=above:{$\scriptstyle{c_k}$}, inner sep=2.5pt, fill=black, circle] at (-4, 3.7)(v1){}; 
\node[label=below:{$\scriptstyle{c_3}$}, inner sep=2.5pt, fill=black, circle] at (-4, -3.7)(v2){}; 
\node[label=above:{$\scriptstyle{c_1}$}, inner sep=2.5pt, fill=black, circle] at (-2, 2)(v3){}; 
\node[label=below:{$\scriptstyle{c_2}$}, inner sep=2.5pt, fill=black, circle] at (-2, -2)(v4){};

\node[label=above:{$\scriptstyle{d_1}$}, inner sep=2.5pt, fill=black, circle] at (2, 2)(v9){}; 
\node[label=below:{$\scriptstyle{d_2}$}, inner sep=2.5pt, fill=black, circle] at (2, -2)(v10){}; 
\node[label=above:{$\scriptstyle{d_m}$}, inner sep=2.5pt, fill=black, circle] at (4, 3.7)(v11){}; 
\node[label=below:{$\scriptstyle{d_3}$}, inner sep=2.5pt, fill=black, circle] at (4, -3.7)(v12){}; 

\node[inner sep=2.5pt, fill=white, circle] at (0, -7.4)(v21){}; 

\draw[black, dotted, thick] (v1)  .. controls +(-6,0) and +(-6,0) .. (v2);
\draw[black, thick] (v1) -- (v3);
\draw[black, thick] (v2) -- (v4);
\draw[black, thick] (v3) -- (v4);
\draw[black, thick] (v3) -- (v10);
\draw[black, thick] (v4) -- (v9);
\draw[black, thick] (v4) -- (v10);
\draw[black, thick] (v9) -- (v10);
\draw[black, thick] (v9) -- (v11);
\draw[black, thick] (v10) -- (v12);
\draw[black, dotted, thick] (v11)  .. controls +(6,0) and +(6,0) .. (v12);

\node at (-5.2, 0) {\scriptsize{even}};
\node at (5.2, 0) {\scriptsize{even}};

\end{tikzpicture}
\hspace{0.2cm}
\begin{tikzpicture}[scale=0.26]

\node[label=above:{$\scriptstyle{c_k}$}, inner sep=2.5pt, fill=black, circle] at (-4, 3.7)(v1){}; 
\node[label=below:{$\scriptstyle{c_3}$}, inner sep=2.5pt, fill=black, circle] at (-4, -3.7)(v2){}; 
\node[label=above:{$\scriptstyle{c_1}$}, inner sep=2.5pt, fill=black, circle] at (-2, 2)(v3){}; 
\node[label=below:{$\scriptstyle{c_2}$}, inner sep=2.5pt, fill=black, circle] at (-2, -2)(v4){};

\node[label=above:{$\scriptstyle{d_2}$}, inner sep=2.5pt, fill=black, circle] at (2, 2)(v9){}; 
\node[label=below:{$\scriptstyle{d_1}$}, inner sep=2.5pt, fill=black, circle] at (2, -2)(v10){}; 
\node[label=above:{$\scriptstyle{d_3}$}, inner sep=2.5pt, fill=black, circle] at (4, 3.7)(v11){}; 
\node[label=below:{$\scriptstyle{d_m}$}, inner sep=2.5pt, fill=black, circle] at (4, -3.7)(v12){}; 

\node[inner sep=2.5pt, fill=white, circle] at (0, -7.4)(v21){}; 

\draw[black, dotted, thick] (v1)  .. controls +(-6,0) and +(-6,0) .. (v2);
\draw[black, thick] (v1) -- (v3);
\draw[black, thick] (v2) -- (v4);
\draw[black, thick] (v3) -- (v4);
\draw[black, thick] (v4) -- (v10);
\draw[black, thick] (v9) -- (v10);
\draw[black, thick] (v9) -- (v11);
\draw[black, thick] (v10) -- (v12);
\draw[black, dotted, thick] (v11)  .. controls +(6,0) and +(6,0) .. (v12);

\node at (-5.2, 0) {\scriptsize{even}};
\node at (5.2, 0) {\scriptsize{even}};

\end{tikzpicture}
\hspace{-0.3cm}
\begin{tikzpicture}[scale=0.28]

\node[label=above:{$\scriptstyle{c_k}$}, inner sep=2.5pt, fill=black, circle] at (-4, 2.5)(v1){}; 
\node[inner sep=2.5pt, fill=black, circle] at (-4, -3.5)(v2){}; 
\node[label=above:{$\scriptstyle{c_1}$}, inner sep=2.5pt, fill=black, circle] at (-2, 2)(v3){}; 
\node[inner sep=2.5pt, fill=black, circle] at (-2, -3)(v4){};

\node at (-4.8,-4.3) {$\scriptstyle{c_3}$};
\node at (-2.6,-2) {$\scriptstyle{c_2}$};

\node[label=right:{$\scriptstyle{d_1}$}, inner sep=2.5pt, fill=black, circle] at (3, -3)(v10){}; 

\node[inner sep=2.5pt, fill=black, circle] at (-2.5, -5)(v13){}; 
\node[label=right:{$\scriptstyle{d_m}$}, inner sep=2.5pt, fill=black, circle] at (3.5, -5)(v14){};

\node at (-1,-3.6) {$\scriptstyle{d_2}$};
\node at (-3.4,-5.5) {$\scriptstyle{d_3}$};

\draw[black, dotted, thick] (v1)  .. controls +(-4.5,0) and +(-4.5,0) .. (v2);
\draw[black, thick] (v1) -- (v3);
\draw[black, thick] (v2) -- (v4);
\draw[black, thick] (v3) -- (v4);
\draw[black, thick] (v4) -- (v10);
\draw[black, thick] (v4) -- (v13);
\draw[black, thick] (v10) -- (v14);

\draw[black, dotted, thick] (v13)  .. controls +(0,-4.5) and + (0,-4.5) .. (v14);

\node at (-4.5, -0.5) {\scriptsize{even}};
\node at (0.6, -5.8) {\scriptsize{even}};

\end{tikzpicture}
\begin{tikzpicture}[scale=0.28]

\node[label=above:{$\scriptstyle{c_k}$}, inner sep=2.5pt, fill=black, circle] at (-4, 2.5)(v1){}; 
\node[inner sep=2.5pt, fill=black, circle] at (-4, -3.5)(v2){}; 
\node[label=above:{$\scriptstyle{c_1}$}, inner sep=2.5pt, fill=black, circle] at (-2, 2)(v3){}; 
\node[inner sep=2.5pt, fill=black, circle] at (-2, -3)(v4){};

\node at (-4.8,-4.3) {$\scriptstyle{c_3}$};
\node at (-2.6,-2) {$\scriptstyle{c_2}$};

\node[label=above:{$\scriptstyle{d_2}$}, inner sep=2.5pt, fill=black, circle] at (3, 2)(v9){}; 
\node[inner sep=2.5pt, fill=black, circle] at (3, -3)(v10){}; 
\node[label=above:{$\scriptstyle{d_3}$}, inner sep=2.5pt, fill=black, circle] at (5, 2.5)(v11){}; 
\node[inner sep=2.5pt, fill=black, circle] at (5, -3.5)(v12){}; 

\node at (5.8,-4.3) {$\scriptstyle{d_m}$};
\node at (3.8,-2) {$\scriptstyle{d_1}$};

\node[inner sep=2.5pt, fill=black, circle] at (-2.5, -5)(v13){}; 
\node[inner sep=2.5pt, fill=black, circle] at (3.5, -5)(v14){};

\node at (-1,-3.6) {$\scriptstyle{e_1}$};
\node at (2.3, -3.6) {$\scriptstyle{e_2}$};
\node at (4.5,-5.5) {$\scriptstyle{e_3}$};
\node at (-3.4,-5.5) {$\scriptstyle{e_n}$};

\draw[black, dotted, thick] (v1)  .. controls +(-4.5,0) and +(-4.5,0) .. (v2);
\draw[black, thick] (v1) -- (v3);
\draw[black, thick] (v2) -- (v4);
\draw[black, thick] (v3) -- (v4);
\draw[black, thick] (v4) -- (v10);
\draw[black, thick] (v9) -- (v10);
\draw[black, thick] (v9) -- (v11);
\draw[black, thick] (v10) -- (v12);
\draw[black, thick] (v4) -- (v13);
\draw[black, thick] (v10) -- (v14);
\draw[black, dotted, thick] (v11)  .. controls +(4.5,0) and +(4.5,0) .. (v12);

\draw[black, dotted, thick] (v13)  .. controls +(0,-4.5) and + (0,-4.5) .. (v14);

\node at (-4.5, -0.5) {\scriptsize{even}};
\node at (0.6, -5.8) {\scriptsize{even}};
\node at (5.5, -0.5) {\scriptsize{even}};

\end{tikzpicture}
\hspace{0.3cm}
\begin{tikzpicture}[scale=0.28]

\node[label=above:{$\scriptstyle{c_k}$}, inner sep=2.5pt, fill=black, circle] at (-4, 2.5)(v1){}; 
\node[label=below:{$\scriptstyle{c_3}$}, inner sep=2.5pt, fill=black, circle] at (-4, -3.5)(v2){}; 
\node[label=above:{$\scriptstyle{c_1}$}, inner sep=2.5pt, fill=black, circle] at (-2, 2)(v3){}; 
\node[label=below:{$\scriptstyle{c_2}$}, inner sep=2.5pt, fill=black, circle] at (-2, -3)(v4){};

\node[inner sep=2.5pt, fill=white, circle] at (0, -7.5)(v21){};

\node[label=above:{$\scriptstyle{d_2}$}, inner sep=2.5pt, fill=black, circle] at (3, 2)(v9){}; 
\node[label=below:{$\scriptstyle{d_1}$}, inner sep=2.5pt, fill=black, circle] at (3, -3)(v10){}; 
\node[label=above:{$\scriptstyle{d_3}$}, inner sep=2.5pt, fill=black, circle] at (5, 2.5)(v11){}; 
\node[label=below:{$\scriptstyle{d_m}$}, inner sep=2.5pt, fill=black, circle] at (5, -3.5)(v12){}; 

\node at (-0.2,1.2) {$\scriptstyle{p_1}$};
\node at (1.4,-0.5) {$\scriptstyle{p_t}$};

\node[inner sep=2.5pt, fill=black, circle] at (-0.33, 0.33)(v5){}; 
\node[inner sep=2.5pt, fill=black, circle] at (1.33, -1.33)(v6){};

\draw[black, dotted, thick] (v1)  .. controls +(-4.5,0) and +(-4.5,0) .. (v2);
\draw[black, thick] (v1) -- (v3);
\draw[black, thick] (v2) -- (v4);
\draw[black, thick] (v3) -- (v4);
\draw[black, thick] (v4) -- (v10);
\draw[black, thick] (v9) -- (v10);
\draw[black, thick] (v9) -- (v11);
\draw[black, thick] (v10) -- (v12);
\draw[black, dotted, thick] (v11)  .. controls +(4.5,0) and +(4.5,0) .. (v12);

\draw[black, thick] (v3) -- (v5);
\draw[black, thick] (v6) -- (v10);
\draw[black, dotted, thick] (v5) -- (v6);
\draw[black, thick] (v4) -- (v5);
\draw[black, thick] (v4) -- (v6);

\node at (-4.5, -0.5) {\scriptsize{even}};
\node at (5.5, -0.5) {\scriptsize{even}};

\end{tikzpicture}
\end{center}
\vspace{-0.7cm}
\caption{The graphs $A_1, A_2, A_3, A_4, A_5, A_6$ are strongly perfect}
\label{fig:graphs_A_1234567}
\end{figure}

We now introduce three graphs that play a role in constructing new minimal non-strongly-perfect graphs. A \emph{larva} is a graph with vertex set $\{v, c_1, \dots, c_k\}$ where $\{c_1, \dots, c_k\}$ is an even hole and $v$ is adjacent to $c_1, c_2$, as in Figure \ref{fig:evolution}. A \emph{pupa} is a graph with vertex set $\{v, c_1, \dots, c_k, p_1, \dots, p_t\}$ where 
\begin{itemize}
\itemsep0em
\item $\{c_1, \dots, c_k\}$ is an even hole,
\item $v \dd p_1 \dd \dots \dd p_t \dd c_1$ is an odd path,
\item $c_2$ is complete to $\{v, p_1, \dots, p_t, c_1\}$,
\item there is no edge other than the ones specified above.
\end{itemize}

\noindent A \emph{butterfly} is a graph with vertex set $\{v, a_1, a_2, \dots, a_k = b_1, b_2, \dots, b_\ell = c_1, c_2, \dots, c_m\}$ such that 
\begin{itemize}
\itemsep0em
\item $P_1 = a_1 \dd a_2 \dd \dots \dd a_k$ is an even path of length at least two,
\item $P_2 = b_1 \dd b_2 \dd \dots \dd b_\ell$ is an odd path,
\item $P_3 = c_1 \dd c_2 \dd \dots \dd c_m$ is an even path of length at least two,
\item $v$ is complete to $P_2 \cup \{a_1, c_m\}$, and
\item there is no edge other than the ones specified above.
\end{itemize}

\vspace{-0.3cm}

\begin{figure}[h]
\begin{center}
\begin{tikzpicture}[scale=0.25]

\node[label=left:{$\scriptstyle{c_k}$}, inner sep=2.5pt, fill=black, circle] at (0, 0)(v1){}; 
\node[label=right:{$\scriptstyle{c_3}$}, inner sep=2.5pt, fill=black, circle] at (9, 0)(v2){}; 
\node[label=left:{$\scriptstyle{c_1}$}, inner sep=2.5pt, fill=black, circle] at (2, 2)(v3){}; 
\node[label=right:{$\scriptstyle{c_2}$}, inner sep=2.5pt, fill=black, circle] at (7, 2)(v4){};
\node[label=above:{$\scriptstyle{v}$}, inner sep=2.5pt, fill=black, circle] at (4.5, 6)(v5){}; 

\draw[black, dotted, thick] (v1)  .. controls +(0,-6.5) and +(0,-6.5) .. (v2);
\draw[black, thick] (v1) -- (v3);
\draw[black, thick] (v2) -- (v4);
\draw[black, thick] (v3) -- (v4);
\draw[black, thick] (v3) -- (v5);
\draw[black, thick] (v4) -- (v5);

\node at (4.5,-1) {\small{even}};

\end{tikzpicture}
\hspace{1cm}
\begin{tikzpicture}[scale=0.25]

\node[label=left:{$\scriptstyle{c_k}$}, inner sep=2.5pt, fill=black, circle] at (0, 0)(v1){}; 
\node[label=right:{$\scriptstyle{c_3}$}, inner sep=2.5pt, fill=black, circle] at (9, 0)(v2){}; 
\node[label=left:{$\scriptstyle{c_1}$}, inner sep=2.5pt, fill=black, circle] at (2, 2)(v3){}; 
\node[label=right:{$\scriptstyle{c_2}$}, inner sep=2.5pt, fill=black, circle] at (7, 2)(v4){};
\node[label=above:{$\scriptstyle{v}$}, inner sep=2.5pt, fill=black, circle] at (4.5, 6)(v5){}; 

\node[label=left:{$\scriptstyle{p_1}$}, inner sep=2.5pt, fill=black, circle] at (3.75, 4.8)(v6){};
\node[label=left:{$\scriptstyle{p_t}$}, inner sep=2.5pt, fill=black, circle] at (2.75, 3.2)(v7){}; 

\draw[black, dotted, thick] (v1)  .. controls +(0,-6.5) and +(0,-6.5) .. (v2);
\draw[black, thick] (v1) -- (v3);
\draw[black, thick] (v2) -- (v4);
\draw[black, thick] (v3) -- (v4);
\draw[black, thick] (v4) -- (v5);
\draw[black, thick] (v4) -- (v6);
\draw[black, thick] (v4) -- (v7);
\draw[black, thick] (v5) -- (v6);
\draw[black, thick] (v3) -- (v7);
\draw[black, dotted, thick] (v6) -- (v7);

\draw[black, thick] (v4) -- (4.8, 3.5);
\draw[black, thick] (v4) -- (4.8, 3);

\node at (4.5,-1) {\small{even}};

\end{tikzpicture}
\hspace{1cm}
\begin{tikzpicture}[scale=0.21]

\node[label=above:{$\scriptstyle{v}$}, inner sep=2.5pt, fill=black, circle] at (0, 0)(v){};

\node[label=right:{$\scriptstyle{c_m}$}, inner sep=2.5pt, fill=black, circle] at (8, -1)(cm){};
\node[label=left:{$\scriptstyle{a_1}$}, inner sep=2.5pt, fill=black, circle] at (-8, -1)(a1){};

\node[label=below:{$\stackbelow{\scriptstyle{a_k}}{\scriptstyle{b_1}}$}, inner sep=2.5pt, fill=black, circle] at (-4, -7.4)(ak){};
\node[label=below:{$\stackbelow{\scriptstyle{c_1}}{\scriptstyle{b_\ell}}$}, inner sep=2.5pt, fill=black, circle] at (4, -7.4)(c1){};

\draw[black, dotted, thick] (a1)  .. controls +(0,-5) and +(0.2,-0.1) .. (ak);

\draw[black, thick] (ak)  .. controls +(1.5,-1) and +(-1.5,-1) .. (c1);

\draw[black, dotted, thick] (c1)  .. controls +(0.1,0.02) and +(0,-5) .. (cm);

\draw[black, thick] (v) -- (a1);
\draw[black, thick] (v) -- (cm);
\draw[black, thick] (v) -- (ak);
\draw[black, thick] (v) -- (c1);
\draw[black, thick] (v) -- (-1.4, -5);
\draw[black, thick] (v) -- (-0.7, -5.5);
\draw[black, thick] (v) -- (0, -6);
\draw[black, thick] (v) -- (0.7, -5.5);
\draw[black, thick] (v) -- (1.4, -5);

\node at (-4.6,-3.1) {\small{even}};
\node at (4.6,-3.1) {\small{even}};

\node at (-8,-6) {$\scriptstyle{P_1}$};
\node at (0,-9.7) {$\scriptstyle{P_2}$};
\node at (8,-6) {$\scriptstyle{P_3}$};

\end{tikzpicture}
\end{center}
\vspace{-0.7cm}
\caption{A larva, a pupa, and a butterfly}
\label{fig:evolution}
\end{figure}
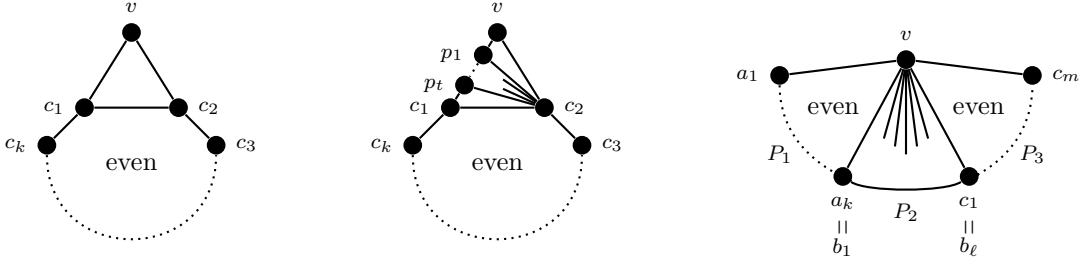

We call $v$ the \emph{head} of the larva (resp. pupa, butterfly). The edges $vc_1$, $vc_2$ of a larva $D$ are called \emph{the side edges} of $D$. Larvas, pupas, and butterflies are strongly perfect as they have a strong basis and a strong stable set.


\section{New minimal non-strongly-perfect graphs} \label{sec:mnsp_graphs}

In this section we present different ways of obtaining new minimal non-strongly-perfect graphs.

\subsection{Desirable/Undesirable heads} \label{subsec:un_desirable}

Let $G$ be a strongly perfect graph and let $v \in V(G)$. We say that $v$ is \emph{wanted} in $G$ if $v \in S$ for every strong stable set $S$ in $G$, \emph{unwanted} in $G$ if $v \notin S$ for every strong stable set $S$ in $G$, and \emph{forced} in $G$ if $v$ is wanted or unwanted in $G$. Let us say that $v$ is \emph{desirable} (resp. \emph{undesirable}) in $G$ if $v$ is wanted (resp. unwanted) in $G$ and is not forced in $G \setminus \{u\}$ for every $u \in V(G)$, $u \neq v$.

\begin{observation}
Let $D$ be a larva or a pupa and let $v$ be the head of $D$. Then, $v$ is undesirable. Let $T$ be a butterfly and let $u$ be the head of $T$. Then, $u$ is desirable.
\label{obs:desirable_undesirable}
\end{observation}

For $i=1,2$, let $G_i$ be a strongly perfect graph and let $v_i \in V(G_i)$ be a desirable or an undesirable vertex. If $v_1, v_2$ are both desirable or both undesirable, one can obtain a minimal non-strongly-perfect graph by connecting $v_1$ and $v_2$ via an odd path. If one is desirable and the other is undesirable, one can obtain a minimal non-strongly-perfect graph by connecting $v_1$ and $v_2$ via an even path. We note that Graph II shown in Figure \ref{fig:minimal_non_strongly_perfect} is obtained this way. In view of Observation \ref{obs:desirable_undesirable}, instead of connecting two larvas via an odd path, we connect two pupas via an odd path, or two butterflies via an odd path, or a pupa and a butterfly via an even path, as shown in Figure \ref{fig:new_mnsp_graphs}.

\begin{figure}[h]
\begin{center}
\begin{tikzpicture}[scale=0.24]

\node[label=above:{$\scriptstyle{c_k}$}, inner sep=2.5pt, fill=black, circle] at (-4, 4)(v1){}; 
\node[label=below:{$\scriptstyle{c_3}$}, inner sep=2.5pt, fill=black, circle] at (-4, -4)(v2){}; 
\node[label=above:{$\scriptstyle{c_1}$}, inner sep=2.5pt, fill=black, circle] at (-2, 2.6)(v3){}; 
\node[label=below:{$\scriptstyle{c_2}$}, inner sep=2.5pt, fill=black, circle] at (-2, -2.6)(v4){};

\node[label=above:{$\scriptstyle{p_1}$}, inner sep=2.5pt, fill=black, circle] at (-0.67, 1.73)(v13){}; 
\node[label=above:{$\scriptstyle{p_t}$}, inner sep=2.5pt, fill=black, circle] at (0.67, 0.86)(v14){}; 

\node[label=above:{$\scriptstyle{v_1}$}, inner sep=2.5pt, fill=black, circle] at (2, 0)(v5){}; 
\node[label=above:{$\scriptstyle{r_1}$}, inner sep=2.5pt, fill=black, circle] at (5, 0)(v6){}; 
\node[label=above:{$\scriptstyle{r_n}$}, inner sep=2.5pt, fill=black, circle] at (8, 0)(v7){}; 
\node[label=above:{$\scriptstyle{v_2}$}, inner sep=2.5pt, fill=black, circle] at (11, 0)(v8){};

\node[label=above:{$\scriptstyle{q_\ell}$}, inner sep=2.5pt, fill=black, circle] at (12.33, 0.86)(v15){}; 
\node[label=above:{$\scriptstyle{q_1}$}, inner sep=2.5pt, fill=black, circle] at (13.67, 1.73)(v16){}; 

\node[label=above:{$\scriptstyle{d_2}$}, inner sep=2.5pt, fill=black, circle] at (15, 2.6)(v9){}; 
\node[label=below:{$\scriptstyle{d_1}$}, inner sep=2.5pt, fill=black, circle] at (15, -2.6)(v10){}; 
\node[label=above:{$\scriptstyle{d_3}$}, inner sep=2.5pt, fill=black, circle] at (17, 4)(v11){}; 
\node[label=below:{$\scriptstyle{d_m}$}, inner sep=2.5pt, fill=black, circle] at (17, -4)(v12){}; 

\draw[black, dotted, thick] (v1)  .. controls +(-6,0) and +(-6,0) .. (v2);
\draw[black, thick] (v1) -- (v3);
\draw[black, thick] (v2) -- (v4);
\draw[black, thick] (v3) -- (v4);
\draw[black, thick] (v4) -- (v5);
\draw[black, thick] (v5) -- (v6);
\draw[black, dotted, thick] (v6) -- (v7);
\draw[black, thick] (v7) -- (v8);
\draw[black, thick] (v8) -- (v10);
\draw[black, thick] (v9) -- (v10);
\draw[black, thick] (v9) -- (v11);
\draw[black, thick] (v10) -- (v12);
\draw[black, dotted, thick] (v11)  .. controls +(6,0) and +(6,0) .. (v12);

\draw[black, thick] (v3) -- (v13);
\draw[black, thick] (v5) -- (v14);
\draw[black, thick] (v4) -- (v13);
\draw[black, thick] (v4) -- (v14);
\draw[black, thick] (v4) -- (-0.66, 0);
\draw[black, dotted, thick] (v13) -- (v14);

\draw[black, thick] (v9) -- (v16);
\draw[black, thick] (v8) -- (v15);
\draw[black, thick] (v10) -- (v16);
\draw[black, thick] (v10) -- (v15);
\draw[black, thick] (v10) -- (13.67, 0);
\draw[black, dotted, thick] (v15) -- (v16);

\node at (-5,0) {\small{even}};
\node at (6.5, -1.5) {$\xleftarrow{\makebox[0.5cm]{}}$ {\small{odd}} $\xrightarrow{\makebox[0.5cm]{}}$};
\node at (18,0) {\small{even}};

\end{tikzpicture}
\hspace{0.4cm}
\begin{tikzpicture}[scale=0.125]

\node[inner sep=2.5pt, fill=black, circle] at (0, 0)(v){};

\node[inner sep=2.5pt, fill=black, circle] at (-1, -8)(cm){};
\node[inner sep=2.5pt, fill=black, circle] at (-1, 8)(a1){};

\node[inner sep=2.5pt, fill=black, circle] at (-7.4, 4)(ak){};
\node[inner sep=2.5pt, fill=black, circle] at (-7.4, -4)(c1){};

\node[inner sep=2.5pt, fill=white, circle] at (5, -11.5)(v21){}; 

\draw[black, dotted, thick] (a1)  .. controls +(-4,0) and +(0.6,2) .. (ak);
\draw[black, thick] (ak)  .. controls +(-1.1,-2) and +(-1.1,2) .. (c1);
\draw[black, dotted, thick] (c1)  .. controls +(0.6,-2) and +(-4,0) .. (cm);

\draw[black, thick] (v) -- (a1);
\draw[black, thick] (v) -- (cm);
\draw[black, thick] (v) -- (ak);
\draw[black, thick] (v) -- (c1);
\draw[black, thick] (v) -- (-5, 1.4);
\draw[black, thick] (v) -- (-5.5, 0.7);
\draw[black, thick] (v) -- (-6, 0);
\draw[black, thick] (v) -- (-5.5, -0.7);
\draw[black, thick] (v) -- (-5, -1.4);

\node[inner sep=2.5pt, fill=black, circle] at (25, 0)(v2){};

\node[inner sep=2.5pt, fill=black, circle] at (26, -8)(cm2){};
\node[inner sep=2.5pt, fill=black, circle] at (26, 8)(a12){};

\node[inner sep=2.5pt, fill=black, circle] at (32.4, 4)(ak2){};
\node[inner sep=2.5pt, fill=black, circle] at (32.4, -4)(c12){};

\draw[black, dotted, thick] (a12)  .. controls +(4,0) and +(-0.6,2) .. (ak2);
\draw[black, thick] (ak2)  .. controls +(1.1,-2) and +(1.1,2) .. (c12);
\draw[black, dotted, thick] (c12)  .. controls +(-0.6,-2) and +(4,0) .. (cm2);

\draw[black, thick] (v2) -- (a12);
\draw[black, thick] (v2) -- (cm2);
\draw[black, thick] (v2) -- (ak2);
\draw[black, thick] (v2) -- (c12);
\draw[black, thick] (v2) -- (30, 1.4);
\draw[black, thick] (v2) -- (30.5, 0.7);
\draw[black, thick] (v2) -- (31, 0);
\draw[black, thick] (v2) -- (30.5, -0.7);
\draw[black, thick] (v2) -- (30, -1.4);

\node[inner sep=2.5pt, fill=black, circle] at (5, 0)(v'){};
\node[inner sep=2.5pt, fill=black, circle] at (20, 0)(v2'){};

\draw[black, thick] (v) -- (v');
\draw[black, thick] (v2) -- (v2');
\draw[black, dotted, thick] (v') -- (v2');

\node at (12.5,-2) {$\xleftarrow{\makebox[0.85cm]{}}$ {\small{odd}} $\xrightarrow{\makebox[0.85cm]{}}$};

\end{tikzpicture}
\begin{tikzpicture}[scale=0.245]

\node[inner sep=2.5pt, fill=black, circle] at (-4, 4)(v1){}; 
\node[inner sep=2.5pt, fill=black, circle] at (-4, -4)(v2){}; 
\node[inner sep=2.5pt, fill=black, circle] at (-2, 2.6)(v3){}; 
\node[inner sep=2.5pt, fill=black, circle] at (-2, -2.6)(v4){};

\node[inner sep=2.5pt, fill=black, circle] at (-0.67, 1.73)(v13){}; 
\node[inner sep=2.5pt, fill=black, circle] at (0.67, 0.86)(v14){}; 

\node[inner sep=2.5pt, fill=black, circle] at (2, 0)(v5){}; 
\node[inner sep=2.5pt, fill=black, circle] at (5, 0)(v6){}; 
\node[inner sep=2.5pt, fill=black, circle] at (8, 0)(v7){};

\draw[black, dotted, thick] (v1)  .. controls +(-6,0) and +(-6,0) .. (v2);
\draw[black, thick] (v1) -- (v3);
\draw[black, thick] (v2) -- (v4);
\draw[black, thick] (v3) -- (v4);
\draw[black, thick] (v4) -- (v5);
\draw[black, thick] (v5) -- (v6);
\draw[black, dotted, thick] (v6) -- (v7);

\draw[black, thick] (v3) -- (v13);
\draw[black, thick] (v5) -- (v14);
\draw[black, thick] (v4) -- (v13);
\draw[black, thick] (v4) -- (v14);
\draw[black, thick] (v4) -- (-0.66, 0);
\draw[black, dotted, thick] (v13) -- (v14);

\node at (-5,0) {\small{even}};
\node at (6.5, -1.5) {$\xleftarrow{\makebox[0.4cm]{}}$ {\small{even}} $\xrightarrow{\makebox[0.4cm]{}}$};

\node[inner sep=2.5pt, fill=black, circle] at (11, 0)(v8){};

\node[inner sep=2.5pt, fill=black, circle] at (12, -4)(cm2){};
\node[inner sep=2.5pt, fill=black, circle] at (12, 4)(a12){};

\node[inner sep=2.5pt, fill=black, circle] at (15.6, 2.7)(ak2){};
\node[inner sep=2.5pt, fill=black, circle] at (15.6, -2.7)(c12){};

\draw[black, thick] (v7) -- (v8);

\draw[black, dotted, thick] (a12)  .. controls +(1,0) and +(-1,1) .. (ak2);
\draw[black, thick] (ak2)  .. controls +(1.3,-1.7) and +(1.3, 1.7) .. (c12);
\draw[black, dotted, thick] (c12)  .. controls +(-1,-1) and +(1,0) .. (cm2);

\draw[black, thick] (v8) -- (a12);
\draw[black, thick] (v8) -- (cm2);
\draw[black, thick] (v8) -- (ak2);
\draw[black, thick] (v8) -- (c12);
\draw[black, thick] (v8) -- (14, 1.1);
\draw[black, thick] (v8) -- (14.5, 0.55);
\draw[black, thick] (v8) -- (15, 0);
\draw[black, thick] (v8) -- (14.5, -0.55);
\draw[black, thick] (v8) -- (14, -1.1);

\end{tikzpicture}
\end{center}
\vspace{-0.4cm}
\caption{New minimal non-strongly-perfect graphs}
\label{fig:new_mnsp_graphs}
\end{figure}
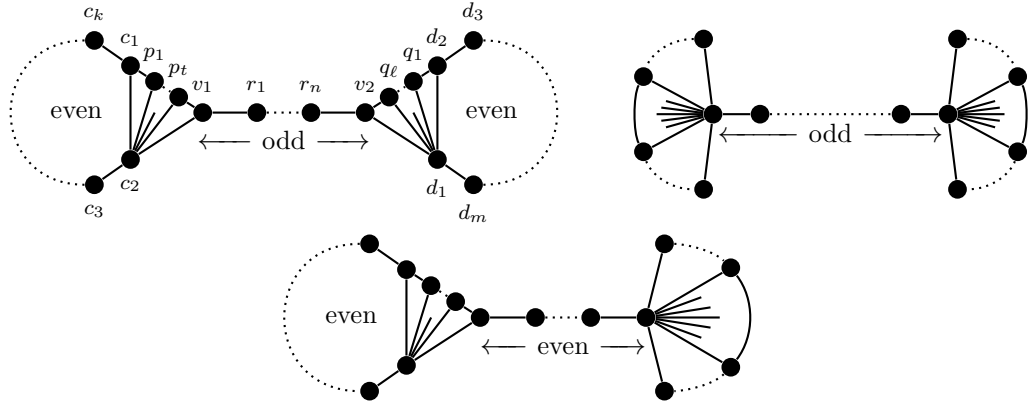

\begin{proposition}
The graphs in Figure \ref{fig:new_mnsp_graphs} are minimal non-strongly-perfect.
\label{prop:new_desirable}
\end{proposition}

\begin{proof}
Let $G$ be the graph obtained by connecting the heads $v_1, v_2$ of two pupas $C_1, C_2$ via an odd path $R = v_1 \dd r_1 \dd r_2 \dd \dots \dd r_n \dd v_2$, as in Figure \ref{fig:new_mnsp_graphs}. We first show that $G$ does not have a strong stable set. Assume for a contradiction that $S$ is a strong stable set in $G$. For $i = 1, 2$, since $v_i$ is a clique cutset satisfying the assumption of Lemma \ref{lem:clique_cutset_sss}, by Observation \ref{obs:desirable_undesirable} and Lemma \ref{lem:clique_cutset_sss}, it follows that $v_1, v_2 \notin S$. Since $\{v_1, r_1\}$, $\{r_1, r_2\}$, \dots, $\{r_{n-1}, r_n\}$ are all maximal cliques of $G$, we have $r_1, r_3, \dots, r_{n-1} \in S$. Then, $r_n \notin S$, a contradiction since $r_n, v_2 \notin S$ but $\{r_n, v_2\}$ is a maximal clique in $G$.

Next, we show that $G$ has a strong basis. Let $C = c_1 \dd c_2 \dd \dots \dd c_k \dd c_1$ and $D = d_1 \dd d_2 \dd \dots \dd d_m \dd d_1$ denote the even holes in $G$, and let $P = p_1 \dd \dots \dd p_t$, $Q = q_1 \dd \dots \dd q_\ell$, and $R = v_1 \dd r_1 \dd \dots \dd r_n \dd v_2$ denote the odd paths in $G$. Then, the graphs $J = G[V(C) \cup V(R) \cup V(D)]$ and $F = G[V(J) \cup V(Q)]$ are strongly perfect since the set $\{c_1, c_3, \dots, c_{k-1}, v_1, r_2, r_4, \dots, r_n, d_1, d_3, \dots, d_{m-1}\}$ is a strong stable set in both $J$ and $F$, and $J$ and $F$ have a strong basis as $\mathcal{B}(J) = \{A_1, 2A_1\}$, and $\mathcal{B}(F) = \{A_1, 2A_1, J\}$. Now, $G$ has a strong basis since $\mathcal{B}(G) = \{A_1, 2A_1, J, F\}$, hence $G$ is minimal non-strongly-perfect. The proof is similar for the other two graphs in Figure \ref{fig:new_mnsp_graphs}, and we leave it to the reader to check.
\end{proof}

\subsection{Evolution of a larva} \label{subsec:evolution}
\setcounter{equation}{0}

\emph{Subdividing} an edge $uv$ means deleting the edge $uv$, adding a new vertex $w$, and adding two new edges $uw$ and $wv$. Let $D$ be a larva with vertices labeled as in Figure \ref{fig:evolution}. \emph{Evolution} is the following operation: subdivide a side edge of $D$, say $vc_1$, an even number of times (i.e., replace $vc_1$ with an odd path of length at least three from $v$ to $c_1$), and make $c_2$ complete to the new vertices. So, a pupa $T$ is obtained from a larva $D$ by evolution. We say that $T$ is obtained from $D$ by \emph{evolving} the side edge $vc_1$. 

Let $H$ be a minimal non-strongly-perfect graph that contains a larva $D$. In $H$, we would like to evolve a side edge of $D$ with the hope of obtaining a new minimal non-strongly-perfect graph $G$. When applicable, we allow several applications of evolution to different side edges of different larvas in $H$. We say that a graph $G$ is an \emph{emanation} of $H$ if $G$ can be obtained from $H$ in this way. For instance, the graph obtained by connecting two pupas via an odd path as in Figure \ref{fig:new_mnsp_graphs} is an emanation of Graph II shown in Figure \ref{fig:minimal_non_strongly_perfect}. We show that emanations of Graphs III, V, and VI, shown in Figure \ref{fig:new_MNSP_evolve}, yield new minimal non-strongly-perfect graphs.

\begin{proposition}
The graphs in Figure \ref{fig:new_MNSP_evolve} are minimal non-strongly-perfect.
\label{prop:new_evolved_ones}
\end{proposition}

\begin{proof}
Let $G_1, G_2, G_3$ be the graphs shown in Figure \ref{fig:new_MNSP_evolve}, from left to right respectively. Note that in $G_3$, we have $e_1 = c_2$, $e_2 = d_1$, and we label the vertices that are not labeled in $G_3$ as follows. Let the odd path from $c_1$ to $d_1$ be $c_1 \dd p_1 \dd \dots \dd p_t \dd d_1$, and let the odd path from $c_2$ to $d_2$ be $c_2 \dd q_1 \dd \dots \dd q_\ell \dd d_2$.

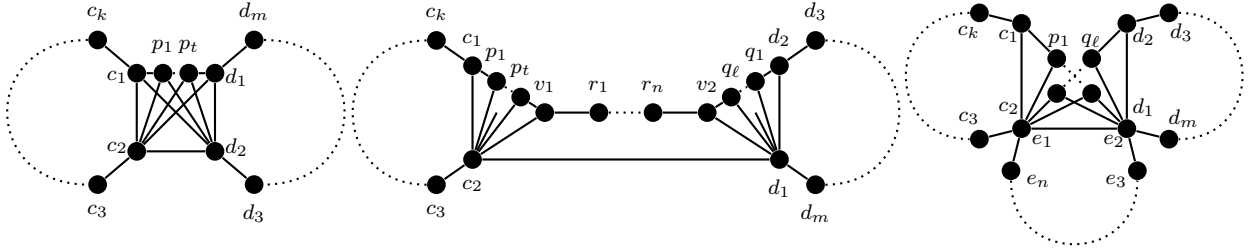
\begin{figure}[h]
\begin{center}
\begin{tikzpicture}[scale=0.26]

\node[label=above:{$\scriptstyle{c_k}$}, inner sep=2.5pt, fill=black, circle] at (-4, 3.7)(v1){}; 
\node[label=below:{$\scriptstyle{c_3}$}, inner sep=2.5pt, fill=black, circle] at (-4, -3.7)(v2){}; 
\node[inner sep=2.5pt, fill=black, circle] at (-2, 2)(v3){}; 
\node[inner sep=2.5pt, fill=black, circle] at (-2, -2)(v4){};

\node at (-3,1.8) {$\scriptstyle{c_1}$};
\node at (-3,-1.8) {$\scriptstyle{c_2}$};

\node[inner sep=2.5pt, fill=black, circle] at (2, 2)(v9){}; 
\node[inner sep=2.5pt, fill=black, circle] at (2, -2)(v10){}; 
\node[label=above:{$\scriptstyle{d_m}$}, inner sep=2.5pt, fill=black, circle] at (4, 3.7)(v11){}; 
\node[label=below:{$\scriptstyle{d_3}$}, inner sep=2.5pt, fill=black, circle] at (4, -3.7)(v12){}; 

\node at (3.1,1.8) {$\scriptstyle{d_1}$};
\node at (3.1,-1.8) {$\scriptstyle{d_2}$};

\node[label=above:{$\scriptstyle{p_1}$}, inner sep=2.5pt, fill=black, circle] at (-0.66, 2)(v5){}; 
\node[label=above:{$\scriptstyle{p_t}$}, inner sep=2.5pt, fill=black, circle] at (0.66, 2)(v6){};

\node[inner sep=2.5pt, fill=white, circle] at (0, -7.4)(v21){}; 

\draw[black, dotted, thick] (v1)  .. controls +(-6,0) and +(-6,0) .. (v2);
\draw[black, thick] (v1) -- (v3);
\draw[black, thick] (v2) -- (v4);
\draw[black, thick] (v3) -- (v4);
\draw[black, thick] (v3) -- (v10);
\draw[black, thick] (v4) -- (v9);
\draw[black, thick] (v4) -- (v10);
\draw[black, thick] (v9) -- (v10);
\draw[black, thick] (v9) -- (v11);
\draw[black, thick] (v10) -- (v12);
\draw[black, dotted, thick] (v11)  .. controls +(6,0) and +(6,0) .. (v12);

\draw[black, thick] (v3) -- (v5);
\draw[black, thick] (v6) -- (v9);
\draw[black, dotted, thick] (v5) -- (v6);
\draw[black, thick] (v4) -- (v5);
\draw[black, thick] (v4) -- (v6);
\draw[black, thick] (v10) -- (v5);
\draw[black, thick] (v10) -- (v6);

\end{tikzpicture}
\hspace{-0.85cm}
\begin{tikzpicture}[scale=0.24]

\node[label=above:{$\scriptstyle{c_k}$}, inner sep=2.5pt, fill=black, circle] at (-4, 4)(v1){}; 
\node[label=below:{$\scriptstyle{c_3}$}, inner sep=2.5pt, fill=black, circle] at (-4, -4)(v2){}; 
\node[label=above:{$\scriptstyle{c_1}$}, inner sep=2.5pt, fill=black, circle] at (-2, 2.6)(v3){}; 
\node[label=below:{$\scriptstyle{c_2}$}, inner sep=2.5pt, fill=black, circle] at (-2, -2.6)(v4){};

\node[label=above:{$\scriptstyle{p_1}$}, inner sep=2.5pt, fill=black, circle] at (-0.67, 1.73)(v13){}; 
\node[label=above:{$\scriptstyle{p_t}$}, inner sep=2.5pt, fill=black, circle] at (0.67, 0.86)(v14){}; 

\node[label=above:{$\scriptstyle{v_1}$}, inner sep=2.5pt, fill=black, circle] at (2, 0)(v5){}; 
\node[label=above:{$\scriptstyle{r_1}$}, inner sep=2.5pt, fill=black, circle] at (5, 0)(v6){}; 
\node[label=above:{$\scriptstyle{r_n}$}, inner sep=2.5pt, fill=black, circle] at (8, 0)(v7){}; 
\node[label=above:{$\scriptstyle{v_2}$}, inner sep=2.5pt, fill=black, circle] at (11, 0)(v8){};

\node[label=above:{$\scriptstyle{q_\ell}$}, inner sep=2.5pt, fill=black, circle] at (12.33, 0.86)(v15){}; 
\node[label=above:{$\scriptstyle{q_1}$}, inner sep=2.5pt, fill=black, circle] at (13.67, 1.73)(v16){}; 

\node[label=above:{$\scriptstyle{d_2}$}, inner sep=2.5pt, fill=black, circle] at (15, 2.6)(v9){}; 
\node[label=below:{$\scriptstyle{d_1}$}, inner sep=2.5pt, fill=black, circle] at (15, -2.6)(v10){}; 
\node[label=above:{$\scriptstyle{d_3}$}, inner sep=2.5pt, fill=black, circle] at (17, 4)(v11){}; 
\node[label=below:{$\scriptstyle{d_m}$}, inner sep=2.5pt, fill=black, circle] at (17, -4)(v12){}; 

\node[inner sep=2.5pt, fill=white, circle] at (0, -8)(v21){}; 

\draw[black, dotted, thick] (v1)  .. controls +(-6,0) and +(-6,0) .. (v2);
\draw[black, thick] (v1) -- (v3);
\draw[black, thick] (v2) -- (v4);
\draw[black, thick] (v3) -- (v4);
\draw[black, thick] (v4) -- (v5);
\draw[black, thick] (v5) -- (v6);
\draw[black, dotted, thick] (v6) -- (v7);
\draw[black, thick] (v7) -- (v8);
\draw[black, thick] (v8) -- (v10);
\draw[black, thick] (v9) -- (v10);
\draw[black, thick] (v9) -- (v11);
\draw[black, thick] (v10) -- (v12);
\draw[black, dotted, thick] (v11)  .. controls +(6,0) and +(6,0) .. (v12);

\draw[black, thick] (v3) -- (v13);
\draw[black, thick] (v5) -- (v14);
\draw[black, thick] (v4) -- (v13);
\draw[black, thick] (v4) -- (v14);
\draw[black, thick] (v4) -- (-0.66, 0);
\draw[black, dotted, thick] (v13) -- (v14);

\draw[black, thick] (v9) -- (v16);
\draw[black, thick] (v8) -- (v15);
\draw[black, thick] (v10) -- (v16);
\draw[black, thick] (v10) -- (v15);
\draw[black, thick] (v10) -- (13.67, 0);
\draw[black, dotted, thick] (v15) -- (v16);

\draw[black, thick] (v4) -- (v10);

\end{tikzpicture}
\hspace{-0.8cm}
\begin{tikzpicture}[scale=0.28]

\node[inner sep=2.5pt, fill=black, circle] at (-4, 2.5)(v1){}; 
\node[inner sep=2.5pt, fill=black, circle] at (-4, -3.5)(v2){}; 
\node[inner sep=2.5pt, fill=black, circle] at (-2, 2)(v3){}; 
\node[inner sep=2.5pt, fill=black, circle] at (-2, -3)(v4){};

\node at (-4.5,1.6) {$\scriptstyle{c_k}$};
\node at (-4.5,-2.5) {$\scriptstyle{c_3}$};
\node at (-2.6,1.4) {$\scriptstyle{c_1}$};
\node at (-2.6,-2) {$\scriptstyle{c_2}$};

\node[inner sep=2.5pt, fill=black, circle] at (3, 2)(v9){}; 
\node[inner sep=2.5pt, fill=black, circle] at (3, -3)(v10){}; 
\node[inner sep=2.5pt, fill=black, circle] at (5, 2.5)(v11){}; 
\node[inner sep=2.5pt, fill=black, circle] at (5, -3.5)(v12){}; 

\node at (5.6,1.6) {$\scriptstyle{d_3}$};
\node at (5.7,-2.5) {$\scriptstyle{d_m}$};
\node at (3.8,1.4) {$\scriptstyle{d_2}$};
\node at (3.8,-2) {$\scriptstyle{d_1}$};

\node[inner sep=2.5pt, fill=black, circle] at (-2.5, -5)(v13){}; 
\node[inner sep=2.5pt, fill=black, circle] at (3.5, -5)(v14){};

\node at (-1,-3.6) {$\scriptstyle{e_1}$};
\node at (2.4, -3.6) {$\scriptstyle{e_2}$};
\node at (2.5,-5.5) {$\scriptstyle{e_3}$};
\node at (-1.2,-5.5) {$\scriptstyle{e_n}$};

\node at (-0.2,1.2) {$\scriptstyle{p_1}$};
\node at (1.3,1.2) {$\scriptstyle{q_\ell}$};

\node[inner sep=2.5pt, fill=black, circle] at (-0.33, 0.33)(v5){}; 
\node[inner sep=2.5pt, fill=black, circle] at (1.33, -1.33)(v6){};
 
\node[inner sep=2.5pt, fill=black, circle] at (-0.33, -1.33)(v7){}; 
\node[inner sep=2.5pt, fill=black, circle] at (1.33, 0.33)(v8){}; 

\draw[black, dotted, thick] (v1)  .. controls +(-4.5,0) and +(-4.5,0) .. (v2);
\draw[black, thick] (v1) -- (v3);
\draw[black, thick] (v2) -- (v4);
\draw[black, thick] (v3) -- (v4);
\draw[black, thick] (v4) -- (v10);
\draw[black, thick] (v9) -- (v10);
\draw[black, thick] (v9) -- (v11);
\draw[black, thick] (v10) -- (v12);
\draw[black, thick] (v4) -- (v13);
\draw[black, thick] (v10) -- (v14);
\draw[black, dotted, thick] (v11)  .. controls +(4.5,0) and +(4.5,0) .. (v12);

\draw[black, dotted, thick] (v13)  .. controls +(0,-4.5) and + (0,-4.5) .. (v14);

\draw[black, thick] (v3) -- (v5);
\draw[black, thick] (v6) -- (v10);
\draw[black, dotted, thick] (v5) -- (v6);
\draw[black, thick] (v4) -- (v5);
\draw[black, thick] (v4) -- (v6);

\draw[black, thick] (v9) -- (v8);
\draw[black, thick] (v4) -- (v7);
\draw[black, dotted, thick] (v7) -- (v8);
\draw[black, thick] (v10) -- (v7);
\draw[black, thick] (v10) -- (v8);

\end{tikzpicture}
\end{center}
\vspace{-1cm}
\caption{New minimal non-strongly-perfect graphs}
\label{fig:new_MNSP_evolve}
\end{figure}

For $i=1,2,3$, we first show that $G_i$ does not have a strong stable set. Assume for a contradiction that $G_i$ has a strong stable set $S_i$.
\begin{equation}
\longbox{{\it For $i=1,2,3$, exactly one of $c_1, c_2$ is in $S_i$, exactly one of $d_1, d_2$ is in $S_i$, and exactly one of $e_1, e_2$ is in $S_3$.}}
\label{eq:even_cycles_problem}
\end{equation}

\begin{proof}
We prove that exactly one of $c_1, c_2$ is in $S_1$. The proof is the same for the other assertions. Clearly, not both $c_1$ and $c_2$ is in $S_1$ as $c_1$ is adjacent to $c_2$. Assume $c_1, c_2 \notin S_1$. Since $\{c_2, c_3\}$, $\{c_3, c_4\}$, $\dots$, $\{c_{k-1}, c_k\}$ are all maximal cliques of $G_1$, we have $c_3, c_5, \dots, c_{k-1} \in S_1$. Then $c_k \notin S_1$, a contradiction since $c_1, c_k \notin S_1$ but $\{c_1, c_k\}$ is a maximal clique of $G_1$.
\end{proof}

In $G_1$, by \eqref{eq:even_cycles_problem}, we have $c_1, d_1 \in S_1$. Thus, $p_1, p_t, c_2, d_2 \notin S_1$. As $\{p_1, p_2, c_2, d_2\}$ is a maximal clique of $G_1$, we deduce that $p_2, \in S_1$, and so, $p_3 \notin S_1$. Since $\{p_3, p_4, c_2, d_2\}$ is a maximal clique of $G_1$, we have $p_4, \in S_1$, and so, $p_5 \notin S_1$. Continuing this argument along the path $p_1 \dd \dots \dd p_t$, it follows that $p_2, p_4, \dots, p_{t-2} \in S_1$. Then $p_{t-1} \notin S_1$, a contradiction since $S_1 \cap \{p_{t-1}, p_t, c_2, d_2\} = \emptyset$.

In $G_2$, since $\{v_1, c_2\}$ and $\{v_2, d_1\}$ are clique cutsets, by Lemma \ref{lem:clique_cutset_sss} and Observation \ref{obs:desirable_undesirable}, we have $v_1, v_2 \notin S_2$. As $\{v_1, r_1\}$, $\{r_1, r_2\}$, \dots, $\{r_{n-1}, r_n\}$ are all maximal cliques of $G_2$, we have $r_1, r_3, \dots, r_{n-1} \in S_2$. Then, $r_n \notin S_2$, a contradiction since $r_n, v_2 \notin S_2$ but $\{r_n, v_2\}$ is a maximal clique in $G_2$.

In $G_3$, by \eqref{eq:even_cycles_problem}, we deduce that either $c_1, d_1 \in S_3$, or $c_2, d_2 \in S_3$. By symmetry, we assume $c_1, d_1 \in S_3$. Thus, $p_1, p_t, c_2 \notin S_3$. Then, since $\{p_1, p_2, c_2\}$ and $\{p_t, p_{t-1}, c_2\}$ are maximal cliques of $G_3$, we have $p_2, p_{t-1} \in S_3$. So, $p_3, p_{t-2} \notin S_3$. Then, since $\{p_3, p_4, c_2\}$ and $\{p_{t-2}, p_{t-3}, c_2\}$ are maximal cliques of $G_3$, we have $p_4, p_{t-3} \in S_3$. Continuing this argument along the path $p_1 \dd \dots \dd p_t$, we reach a contradiction since the path $p_1 \dd \dots \dd p_t$ is odd. This completes the proof that $G_i$ does not have a strong stable set.

Next, for $i=1,2,3$, we show that the graph $G_i$ has a strong basis. The graph $G_1$ has a strong basis since $\mathcal{B}(G_1) = \{A_1, A_2\}$. Let $C = c_1 \dd c_2 \dd \dots \dd c_k \dd c_1$ and $D = d_1 \dd d_2 \dd \dots \dd d_m \dd d_1$ denote the even holes in $G_2$, and let $P = p_1 \dd \dots \dd p_t$, $Q = q_1 \dd \dots \dd q_\ell$, and $R = v_1 \dd r_1 \dd \dots \dd r_n \dd v_2$ denote the odd paths in $G_2$. Notice that $T = G_2[V(D) \cup V(Q) \cup V(R) \cup \{c_2\}]$ is a butterfly in $G_2$. Moreover, the graph $F = G_2[V(T) \cup V(C)]$ is strongly perfect because the set $\{d_1, d_3, \dots, d_{m-1}, v_1, r_2, r_4, \dots, r_n, c_1, c_3, \dots, c_{k-1}\}$ is a strong stable set in $F$, and $F$ has a strong basis since $\mathcal{B}(F) = \{A_1, A_3, A_4, A_5, T\}$. Now, $G_2$ has a strong basis since $\mathcal{B}(G_2) = \{A_1, A_3, A_4, A_5, T, F\}$.

Let $C = c_1 \dd c_2 \dd \dots \dd c_k \dd c_1$, $D = d_1 \dd d_2 \dd \dots \dd d_m \dd d_1$, and $E = e_1 \dd e_2 \dd \dots \dd e_n \dd e_1$ denote the even holes in $G_3$, and let $P = p_1 \dd \dots \dd p_t$ and $Q = q_1 \dd \dots \dd q_\ell$ denote the odd paths in $G_3$. Notice that the graph $T = G_3[V(D) \cup V(E) \cup V(Q)]$ is a butterfly in $G_3$.
\begin{equation}
\longbox{{\it $H = G_3[V(T) \cup V(C)]$ and $J = G_3[V(C) \cup V(D) \cup V(P) \cup V(Q)]$ are strongly perfect.}}
\label{eq:G3_F_H_str_per}
\end{equation}

\begin{proof}
The set $\{d_1, d_3, \dots, d_{m-1}, e_4, e_6, \dots, e_n, c_1, c_3, \dots, c_{k-1}\}$ is a strong stable set in $H$, and $H$ has a strong basis since $\mathcal{B}(H) = \{A_1, A_3, A_4, A_5, A_6, T\}$. Therefore, $H$ is strongly perfect. The set $\{c_1, c_3, \dots, c_{k-1}, p_2, p_4, \dots, p_t, q_1, q_3, \dots, q_{\ell-1}, d_2, d_4, \dots, d_m\}$ is a strong stable set in $J$, and $J$ has a strong basis since $\mathcal{B}(J) = \{A_1, A_3, A_6\}$. Hence, $J$ is strongly perfect.
\end{proof}

Now, $G_3$ has a strong basis since $\mathcal{B}(G_3) = \{A_1, A_3, A_4, A_5, A_6, T, H, J\}$.
\end{proof}

\subsection{Mutation of a larva}
\setcounter{equation}{0}

Let $D$ be a larva with vertices labeled as in Figure \ref{fig:evolution}. As previously described in Section \ref{subsec:evolution}, a pupa $P$ can be obtained from $D$ by evolution. Another way of describing how to obtain a pupa $P$ from $D$ is as follows. Assume that $k \geq 6$ and let $i$ be an odd number such that $5 \leq i \leq k-1$. A pupa $P$ can be obtained from $D$ by making $c_2$ complete to $\{c_i, c_{i+1}, \dots, c_k\}$, i.e., by adding an even number of chords from $c_2$ to its far side of the even cycle in $D$. In this case, we say that $c_2$ is \emph{mutating towards $c_i$}. Notice that $H = c_2 \dd c_3 \dd \dots \dd c_i \dd c_2$ is an even hole after $c_2$ mutates towards $c_i$ (and the new graph is a pupa). We could also add an even number of chords from $c_1$ to its far side of the even cycle, that is, for some even number $j$ with $4 \leq j \leq k-2$, we make $c_1$ complete to $\{c_3, c_4, \dots, c_j\}$, in which case we say $c_1$ is mutating towards $c_j$. However, at most one of $c_1, c_2$ is allowed to mutate.

Let $D$ be a larva (with vertex labels as in Figure \ref{fig:evolution}) in a graph $G_0$. For $j = 1,2,\dots, n$, let $G_j$ be obtained from $G_{j-1}$ by mutating $c_{i_{j-1}}$ towards $c_{i_j}$ where $i_0 =1$ and $4 \leq i_1, i_2, \dots i_n \leq k$. This operation is called \emph{mutation} and we call the graph $G_n$ a \emph{mutated} $G_0$. A mutated pupa and a mutated butterfly are shown in Figure \ref{fig:mutated_butterflies}, where the mutated pupa is obtained by mutating $c_1$ towards $c_4$, then mutating $c_4$ towards $c_{k-1}$, and then mutating $c_{k-1}$ towards $c_6$.

\vspace{-0.3cm}

\begin{figure}[h]
\begin{center}
\begin{tikzpicture}[scale=0.22]

\node[inner sep=2.5pt, fill=black, circle] at (0, 0)(v1){}; 
\node[inner sep=2.5pt, fill=black, circle] at (9, 0)(v2){}; 
\node[inner sep=2.5pt, fill=black, circle] at (2, 2)(v3){}; 
\node[inner sep=2.5pt, fill=black, circle] at (7, 2)(v4){};
\node[label=above:{$\scriptstyle{v}$}, inner sep=2.5pt, fill=black, circle] at (4.5, 6)(v5){}; 

\node at (0.8,2.5) {$\scriptstyle{c_1}$};
\node at (8.2,2.5) {$\scriptstyle{c_2}$};
\node at (9.2,1.5) {$\scriptstyle{c_3}$};
\node at (10.2,0.2) {$\scriptstyle{c_4}$};
\node at (10.2,-1.3) {$\scriptstyle{c_5}$};
\node at (10.2,-2.8) {$\scriptstyle{c_6}$};
\node at (-2,-2.8) {$\scriptstyle{c_{k-3}}$};
\node at (-2,-1.3) {$\scriptstyle{c_{k-2}}$};
\node at (-2,0.2) {$\scriptstyle{c_{k-1}}$};
\node at (-0.2,1.5) {$\scriptstyle{c_k}$};

\node[inner sep=2.5pt, fill=black, circle] at (3.75, 4.8)(v6){};
\node[inner sep=2.5pt, fill=black, circle] at (2.75, 3.2)(v7){}; 

\node[inner sep=2.5pt, fill=black, circle] at (8, 1)(v8){}; 

\node[inner sep=2.5pt, fill=black, circle] at (0, -3)(v9){};
\node[inner sep=2.5pt, fill=black, circle] at (9, -3)(v10){};

\node[inner sep=2.5pt, fill=black, circle] at (1, 1)(v11){}; 
\node[inner sep=2.5pt, fill=black, circle] at (0, -1.5)(v12){};
\node[inner sep=2.5pt, fill=black, circle] at (9, -1.5)(v13){};

\draw[black, dotted, thick] (v9)  .. controls +(0,-6.5) and +(0,-6.5) .. (v10);
\draw[black, thick] (v1) -- (v3);
\draw[black, thick] (v2) -- (v4);
\draw[black, thick] (v3) -- (v4);
\draw[black, thick] (v4) -- (v5);
\draw[black, thick] (v4) -- (v6);
\draw[black, thick] (v4) -- (v7);
\draw[black, thick] (v5) -- (v6);
\draw[black, thick] (v3) -- (v7);
\draw[black, dotted, thick] (v6) -- (v7);

\draw[black, thick] (v4) -- (4.8, 3.5);
\draw[black, thick] (v4) -- (4.8, 3);

\draw[black, thick] (v3) -- (v2);
\draw[black, thick] (v3) -- (v8);

\draw[black, thick] (v1) -- (v9);
\draw[black, thick] (v2) -- (v10);

\draw[black, thick] (v2) -- (v1);
\draw[black, thick] (v2) -- (v11);

\draw[black, thick] (v1) -- (v10);
\draw[black, thick] (v1) -- (v13);

\node at (4.5,-4.5) {\small{even}};

\end{tikzpicture}
\hspace{0.5cm}
\begin{tikzpicture}[scale=0.26]

\node[inner sep=2.5pt, fill=black, circle] at (2, 2)(v3){}; 
\node[inner sep=2.5pt, fill=black, circle] at (6, 2)(v4){};
\node[inner sep=2.5pt, fill=black, circle] at (-1, 2)(v6){}; 
\node[inner sep=2.5pt, fill=black, circle] at (9, 2)(v7){};

\node[inner sep=2.5pt, fill=black, circle] at (4, 5)(v5){}; 

\node[inner sep=2.5pt, fill=black, circle] at (0.5, 2)(v12){}; 

\node[inner sep=2.5pt, fill=black, circle] at (2, 7.5)(v8){};
\node[inner sep=2.5pt, fill=black, circle] at (6, 7.5)(v9){}; 

\node[inner sep=2.5pt, fill=black, circle] at (3, 6.25)(v10){};
\node[inner sep=2.5pt, fill=black, circle] at (5, 6.25)(v11){};

\node[inner sep=2.5pt, fill=white, circle] at (5, -2)(v21){};

\draw[black, dotted, thick] (v6)  .. controls +(-5,1) and +(-5,3) .. (v8);
\draw[black, dotted, thick] (v7)  .. controls +(5,1) and +(5,3) .. (v9);
\draw[black, thick] (v3) -- (v4);
\draw[black, thick] (v3) -- (v5);
\draw[black, thick] (v4) -- (v5);
\draw[black, thick] (v6) -- (v3);
\draw[black, thick] (v7) -- (v4);
\draw[black, thick] (v8) -- (v5);
\draw[black, thick] (v9) -- (v5);

\draw[black, thick] (v3) -- (v8);
\draw[black, thick] (v3) -- (v10);

\draw[black, thick] (v4) -- (v9);
\draw[black, thick] (v4) -- (v11);

\draw[black, thick] (v8) -- (v6);
\draw[black, thick] (v8) -- (v12);

\draw[black, thick] (v5) -- (3.5, 3);
\draw[black, thick] (v5) -- (4, 3);
\draw[black, thick] (v5) -- (4.5, 3);

\node at (8.5,5) {\small{even}};
\node at (-1.3,5.3) {\small{even}};

\node at (4,0.6) {\small{odd}};

\end{tikzpicture}
\end{center}
\vspace{-0.9cm}
\caption{A mutated pupa and a mutated butterfly}
\label{fig:mutated_butterflies}
\end{figure}
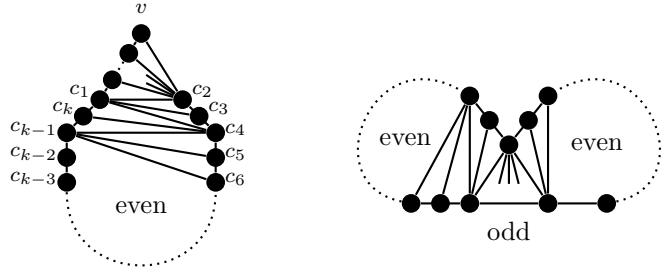

It is immediate to show that mutated pupas and butterflies are strongly perfect. Moreover, Observation \ref{obs:desirable_undesirable} holds also for mutated pupas and butterflies, i.e., the head of a mutated pupa is undesirable, and the head of a mutated butterfly is desirable. This suggests that we can obtain new minimal non-strongly-perfect graphs by connecting mutated pupas and butterflies via an even or odd paths, as before. More generally, we invite the reader to check that the following holds. We omit the proof as it is similar to the proofs of Proposition \ref{prop:new_desirable} and \ref{prop:new_evolved_ones}.

\begin{proposition}
Let $G_1, G_2, \dots, G_6$ be the graphs shown in Figure \ref{fig:new_mnsp_graphs} and Figure \ref{fig:new_MNSP_evolve}. For $i=1,\dots,6$, let $G_i'$ be a mutated $G_i$. Then, $G_i'$ is a minimal non-strongly-perfect graph.
\label{prop:new_mutated_ones}
\end{proposition}

Proposition \ref{prop:new_mutated_ones} provides several new minimal non-strongly-perfect graphs. Other new minimal non-strongly-perfect graphs can be obtained by considering different possible mutations of Graph IV. Two of them are shown in Figure \ref{fig:new_mnsp_chord}. We again omit the proof and leave it to the reader.

\begin{proposition}
The graphs in Figure \ref{fig:new_mnsp_chord} are minimal non-strongly-perfect.
\end{proposition}

\vspace{-0.5cm}

\begin{figure}[h]
\begin{center}
\begin{tikzpicture}[scale=0.23]

\node[inner sep=2.5pt, fill=black, circle] at (0.5, 0)(v1){}; 
\node[inner sep=2.5pt, fill=black, circle] at (7.5, 0)(v2){}; 

\node[inner sep=2.5pt, fill=black, circle] at (2, 2)(v3){}; 
\node[inner sep=2.5pt, fill=black, circle] at (6, 2)(v4){};
\node[inner sep=2.5pt, fill=black, circle] at (-1, 2)(v6){}; 
\node[inner sep=2.5pt, fill=black, circle] at (9, 2)(v7){};

\node[inner sep=2.5pt, fill=black, circle] at (4, 5)(v5){}; 

\node[inner sep=2.5pt, fill=black, circle] at (2, 7.5)(v8){};
\node[inner sep=2.5pt, fill=black, circle] at (6, 7.5)(v9){}; 

\node[inner sep=2.5pt, fill=black, circle] at (3, 6.25)(v10){};
\node[inner sep=2.5pt, fill=black, circle] at (7.5, 2)(v11){};
\node[inner sep=2.5pt, fill=black, circle] at (1.25, 1)(v12){};

\draw[black, dotted, thick] (v1)  .. controls +(0,-5) and +(0,-5) .. (v2);
\draw[black, dotted, thick] (v6)  .. controls +(-5,1) and +(-5,3) .. (v8);
\draw[black, dotted, thick] (v7)  .. controls +(5,1) and +(5,3) .. (v9);
\draw[black, thick] (v1) -- (v3);
\draw[black, thick] (v2) -- (v4);
\draw[black, thick] (v3) -- (v4);
\draw[black, thick] (v3) -- (v5);
\draw[black, thick] (v4) -- (v5);
\draw[black, thick] (v6) -- (v3);
\draw[black, thick] (v7) -- (v4);
\draw[black, thick] (v8) -- (v5);
\draw[black, thick] (v9) -- (v5);

\draw[black, thick] (v3) -- (v8);
\draw[black, thick] (v3) -- (v10);

\draw[black, thick] (v4) -- (v1);
\draw[black, thick] (v4) -- (v12);

\draw[black, thick] (v5) -- (v7);
\draw[black, thick] (v5) -- (v11);

\node at (4.2,-1) {\small{even}};
\node at (8.5,5) {\small{even}};
\node at (-0.5,5) {\small{even}};

\end{tikzpicture}
\hspace{-0.4cm}
\begin{tikzpicture}[scale=0.23]

\node[inner sep=2.5pt, fill=black, circle] at (0.5, 0)(v1){}; 
\node[inner sep=2.5pt, fill=black, circle] at (7.5, 0)(v2){}; 

\node[inner sep=2.5pt, fill=black, circle] at (2, 2)(v3){}; 
\node[inner sep=2.5pt, fill=black, circle] at (6, 2)(v4){};
\node[inner sep=2.5pt, fill=black, circle] at (-1, 2)(v6){}; 
\node[inner sep=2.5pt, fill=black, circle] at (9, 2)(v7){};

\node[inner sep=2.5pt, fill=black, circle] at (4, 5)(v5){}; 

\node[inner sep=2.5pt, fill=black, circle] at (2, 7.5)(v8){};
\node[inner sep=2.5pt, fill=black, circle] at (6, 7.5)(v9){}; 

\node[inner sep=2.5pt, fill=black, circle] at (3, 6.25)(v10){};
\node[inner sep=2.5pt, fill=black, circle] at (5, 6.25)(v11){};
\node[inner sep=2.5pt, fill=black, circle] at (1.25, 1)(v12){};

\draw[black, dotted, thick] (v1)  .. controls +(0,-5) and +(0,-5) .. (v2);
\draw[black, dotted, thick] (v6)  .. controls +(-5,1) and +(-5,3) .. (v8);
\draw[black, dotted, thick] (v7)  .. controls +(5,1) and +(5,3) .. (v9);
\draw[black, thick] (v1) -- (v3);
\draw[black, thick] (v2) -- (v4);
\draw[black, thick] (v3) -- (v4);
\draw[black, thick] (v3) -- (v5);
\draw[black, thick] (v4) -- (v5);
\draw[black, thick] (v6) -- (v3);
\draw[black, thick] (v7) -- (v4);
\draw[black, thick] (v8) -- (v5);
\draw[black, thick] (v9) -- (v5);

\draw[black, thick] (v3) -- (v8);
\draw[black, thick] (v3) -- (v10);

\draw[black, thick] (v4) -- (v1);
\draw[black, thick] (v4) -- (v12);

\draw[black, thick] (v4) -- (v9);
\draw[black, thick] (v4) -- (v11);

\node at (4.2,-1) {\small{even}};
\node at (8.5,5) {\small{even}};
\node at (-0.5,5) {\small{even}};

\end{tikzpicture}
\end{center}
\vspace{-0.75cm}
\caption{New minimal non-strongly-perfect graphs}
\label{fig:new_mnsp_chord}
\end{figure}
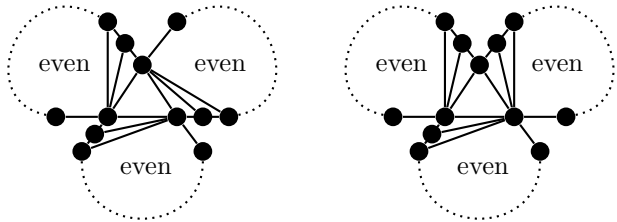

While obtaining a complete list of minimal non-strongly-perfect graphs appears to be out of reach, one might conjecture that the characterization of outerplanar strongly perfect graphs can be obtained through the minimal examples given in this paper.


\end{document}